\input ./style/arxiv-general.cfg
\documentclass[aap,MSNbibl,noautosecdot,dvips]{arximspdf}
\makeatletter
   \@ifpackageloaded{graphicx}{}{\usepackage{graphicx}}
\makeatother
\usepackage{multirow}
\usepackage{clrscode}
\usepackage{dcolumn}

\doi{10.1214/15-AAP1100}
\volume{26}
\issue{1}
\pubyear{2016}
\firstpage{549}
\lastpage{596}
\docsubty{FLA}

\makeatletter
\newcommand{\angler}{\rangle}
\newcommand{\anglel}{\langle}
\newcommand{\rrvert}{\vert}
\newcommand{\llvert}{\vert}
\newcolumntype{d}[1]{D{.}{.}{#1}}
\newtheorem{theorem}{Theorem}

\newtheorem{lemma}[theorem]{Lemma}
\newtheorem{proposition}[theorem]{Proposition}
\newtheorem{corollary}[theorem]{Corollary}
\newproclaim{definition}[theorem]{Definition}
\newproclaim{example}{Example}
\newproclaim{conjecture}[theorem]{Conjecture}
\newproclaim{remark}[theorem]{Remark}
\newproclaim{tsmcmci}{Toric Symplectic Markov Chain Monte Carlo integration}
\makeatother

\begin{document}
\begin{frontmatter}

\title{The symplectic geometry of closed equilateral random walks in 3-space}
\runtitle{Symplectic geometry of closed random walks}

\begin{aug}
\author[A]{\fnms{Jason}~\snm{Cantarella}\ead[label=e1]{jason@math.uga.edu}\thanksref{T1}}
\and
\author[B]{\fnms{Clayton}~\snm{Shonkwiler}\corref{}\ead[label=e2]{clayton@math.colostate.edu}\thanksref{T2}}
\runauthor{J. Cantarella and C. Shonkwiler}
\affiliation{University of Georgia and Colorado State University}
\address[A]{Department of Mathematics\\
University of Georgia\\
Athens, Georgia 30602\\
USA\\
\printead{e1}}
\address[B]{Department of Mathematics\\
Colorado State University \\
Campus Delivery 1874 \\
Fort Collins, Colorado 80523 \\
USA\\
\printead{e2}}
\end{aug}
\thankstext{T1}{Supported in part by the Simons Foundation.}
\thankstext{T2}{Supported in part by the Simons Foundation and the UGA
VIGRE II Grant DMS-07-38586.}

%
\received{\smonth{10} \syear{2013}}
%
\revised{\smonth{1} \syear{2015}}

%
\begin{abstract}
A closed equilateral random walk in 3-space is a selection of unit
length vectors giving the steps of the walk conditioned on the
assumption that the sum of the vectors is zero. The sample space of
such walks with $n$ edges is the $(2n-3)$-dimensional Riemannian
manifold of equilateral closed polygons in~$\mathbb{R}^3$. We study closed
random walks using the symplectic geometry of the $(2n-6)$-dimensional
quotient of the manifold of polygons by the action of the rotation
group $\operatorname{SO}(3)$.

The basic objects of study are the moment maps on equilateral random
polygon space given by the lengths of any $(n-3)$-tuple of
nonintersecting diagonals. The Atiyah--Guillemin--Sternberg theorem
shows that the image of such a moment map is a convex polytope in
$(n-3)$-dimensional space, while the Duistermaat--Heckman theorem shows
that the pushforward measure on this polytope is Lebesgue measure on
$\mathbb{R}^{n-3}$.
Together, these theorems allow us to define a measure-preserving set of
``action-angle'' coordinates on the space of closed equilateral
polygons. The new coordinate system allows us to make explicit
computations of exact expectations for total curvature and for some
chord lengths of closed (and confined) equilateral random walks, to
give statistical criteria for sampling algorithms on the space of
polygons and to prove that the probability that a randomly chosen
equilateral hexagon is unknotted is at least $\frac{1}{2}$.

We then use our methods to construct a new Markov chain sampling
algorithm for equilateral closed polygons, with a simple modification
to sample (rooted) confined equilateral closed polygons. We prove
rigorously that our algorithm converges geometrically to the standard
measure on the space of closed random walks, give a theory of error
estimators for Markov chain Monte Carlo integration using our method
and analyze the performance of our method. Our methods also apply to
open random walks in certain types of confinement, and in general to
walks with arbitrary (fixed) edgelengths as well as equilateral walks.
\end{abstract}

%
\end{frontmatter}

\section{\texorpdfstring{Introduction.}{Introduction}}

In this paper, we consider the classical model of a random walk in
$\mathbb{R}^3$---the walker chooses each step uniformly from the unit sphere. Some
of the first results in the theory of these random walks are based on
the observation that if a point is distributed uniformly on the surface
of a sphere in $3$-space and we write its position in terms of the
cylindrical coordinates $z$ and~$\theta$, then $z$ and~$\theta$ are
independent, uniform random variates. This is usually called
Archimedes' theorem, and it is the underlying idea in the work of Lord
Rayleigh \cite{Rayleigh1919do},  Treloar \cite{TF9464200077}
and many others in the theory of random walks, starting at the
beginning of the 20th century. In particular, it means that the vector
of $z$-coordinates of the edges (steps) of a random walk is uniformly
distributed on a hypercube and that the vector of $\theta$-coordinates
of the edges is uniformly distributed on the $n$-torus.

When we condition the walk on closure, it seems that this pleasant
structure disappears: the individual steps in the walk are no longer
independent random variates, and there are no obvious uniformly
distributed random angles or distances in sight. This makes the study
of closed random walks considerably more difficult than the study of
general random walks. The main point of this paper is that the apparent
disappearance of this structure in the case of closed random walks is
only an illusion. In fact, there is a very similar structure on the
space of closed random walks if we are willing to pay the modest price
of identifying walks related by translation and rigid rotation in
$\mathbb{R}^3$. This structure is less obvious, but just as useful.

As it turns out, Archimedes' theorem was generalized in deep and
interesting ways in the later years of the 20th century, being revealed
as a special case of the Duistermaat--Heckman theorem \cite{Duistermaat1982hq}
for toric symplectic manifolds. Further, Kapovich
and Millson \cite{Kapovich1996p2605} and Hausmann and Knutson
\cite{Knutson2iyExxE} revealed a toric symplectic structure on the
quotient of the space of closed equilateral polygons by the action of
the Euclidean group $E(3)$.
Together, these theorems define a structure on closed random walk space
which is remarkably similar to the structure on the space of open
random walks: if we view an $n$-edge closed equilateral walk as the
boundary of a triangulated surface, we will show below that the lengths
of the $n-3$ diagonals of the triangulation are uniformly distributed
on the polytope given by the triangle inequalities and that the $n-3$
dihedral angles at these diagonals of the triangulated surface are
distributed uniformly and independently on the $(n-3)$-torus. This
structure allows us to define a special set of ``action-angle''
coordinates which provide a measure-preserving map from the product of
a convex polytope $P \subset\mathbb{R}^{n-3}$ and the $(n-3)$-torus (again,
with their standard measures) to a full-measure subset of the
Riemannian manifold of closed polygons of fixed edgelengths.

Understanding this picture allows us to make some new explicit
calculations and prove some new theorems about closed equilateral
random walks. For instance, we are able to find an exact formula for
the total curvature of closed equilateral polygons, to prove that the
expected lengths of chords skipping various numbers of edges are equal
to the coordinates of the center of mass of a certain polytope, to
compute these moments explicitly for random walks with small numbers of
edges and to give a simple proof that at least $1/2$ of equilateral
hexagons are unknotted. Further, we will be able to give a unified
theory of several interesting problems about confined random walks, and
to provide some explicit computations of chordlengths for confined
walks. We state upfront that all the methods we use from symplectic
geometry are by now entirely standard; the new contribution of our
paper lies in the application of these powerful tools to geometric probability.

We will then turn to sampling for the second half of our paper. Our
theory immediately suggests a new Markov chain sampling algorithm for
confined and unconfined random walks. We will show that the theory of
hit-and-run sampling on convex polytopes immediately yields a sampling
algorithm which converges at a geometric rate to the usual probability
measure on equilateral closed random walks (or equilateral closed
random walks in confinement). Geometric convergence allows us to apply
standard Markov Chain Monte Carlo theory to give error estimators for
MCMC integration over the space of closed equilateral random walks
(either confined or unconfined). Our sampling algorithm works for any
toric symplectic manifold, so we state the results in general terms. We
do this primarily because various interesting confinement models for
random walks have a natural toric symplectic structure, though our
results are presumably applicable far outside the theory of random
walks. As with the tools we use from symplectic geometry, hit-and-run
sampling and MCMC error estimators are entirely standard ways to
integrate over convex polytopes. Again, our main contribution is to
show that these powerful tools apply to closed and confined random
walks with fixed edgelengths and to lay out some initial results which
follow from their use.

\section{\texorpdfstring{Toric symplectic manifolds and action-angle coordinates.}{Toric symplectic manifolds and action-angle coordinates}}
\label{secmomentpolytopesampling}

We begin with a capsule summary of some relevant ideas from symplectic
geometry. A symplectic manifold $M$ is a $2n$-dimensional manifold with
a special nondegenerate $2$-form $\omega$ called the \emph{symplectic
form}. The volume form $  \mathrm{d}m= \frac{1}{n!}\omega^n$ on $M$
is called
the \emph{symplectic volume} or \emph{Liouville volume} and the
corresponding measure is called \emph{symplectic measure}. A
diffeomorphism of a symplectic manifold which preserves the symplectic
form is called a \emph{symplectomorphism}; it must preserve symplectic
volume as well. A symmetry of the manifold is a 1-parameter group of
symplectomorphisms; differentiating at the identity yields a vector
field on the manifold giving the velocity of each point as the group
starts to act. For example, rotating the sphere around the $z$-axis
gives a vector field of velocities tangent to the circles of latitude.

\setcounter{footnote}{2}

We can use the 2-form to pair vector fields on $M$ with 1-forms by
contraction: $\vec{v} \mapsto\omega(\vec{v},\cdot)$. We call this
operation $j$. If applying $j$ to the velocity field of a symmetry
yields an exact $1$-form $\mathrm{d}\mu$, the action is called \emph{Hamiltonian}. The primitive $\mu$ of the $1$-form is a function
on~$M$, which must be constant along any integral curve of the velocity
field by construction. This conserved quantity is called the \emph{moment map} of the action $\mu\dvtx M \rightarrow\mathbb{R}$. If $k$ such
symmetries commute,\footnote{Symmetries which do not commute may be
part of the action of a (noncommutative) Lie group on $M$. The moment
map has a different meaning in this case. We will return to this point
later.} they define an action of the torus $T^k$ on $M$. In this case,
the moment map yields a $k$-dimensional vector of conserved quantities,
so the moment map $\mu$ maps $M$ to $\mathbb{R}^k$ (see \cite{CannasdaSilva2001cg}, Part~VIII).

Two powerful theorems apply to the moment maps of Hamiltonian torus
actions. The convexity theorem of Atiyah~\cite{Atiyah1982ih} and
Guillemin--Sternberg~\cite{Guillemin1982gx}
states that the image of the moment map is a convex polytope $P$ in
$\mathbb{R}^k$, which is called the \emph{moment polytope}. Further, the vertices
of the moment polytope are the images under the moment map of the fixed
points of the torus action, allowing one to find the moment polytope in
practice.
Next, if the action is effective, that is, nonidentity elements act nontrivially, the
Duistermaat--Heckman theorem \cite{Duistermaat1982hq}
asserts that the pushforward of symplectic measure to the moment
polytope $P$ is a piecewise polynomial multiple of Lebesgue measure.
If $k$ is half the dimension of~$M$, that is, $k =
n$, the symplectic manifold is called a \emph{toric symplectic
manifold} and the pushforward measure on $P$ is a constant multiple of
Lebesgue measure.

If we can invert the moment map, we can construct a map $\alpha\dvtx P
\times T^n \rightarrow M$ compatible with $\mu$ which parametrizes a
full-measure subset of the $2n$-dimensional manifold $M$ by the $n$
coordinates of points in $P$, which are called the ``action''
variables, and the $n$ angles in~$T^n$, which are called the
corresponding ``angle'' variables. By convention, we call the action
variables $d_i$ and the angle variables $\theta_i$. We have
the following.

\begin{theorem}[(Duistermaat--Heckman \cite{Duistermaat1982hq}, see
Chapter~30 of \cite{CannasdaSilva2001cg})]
\label{theoremmp}
Suppose $M$ is a $2n$-dimensional toric symplectic manifold with moment
polytope $P$, $T^n$ is the $n$-torus ($n$ copies of the circle) and
$\alpha$ inverts the moment map. If we take the standard measure on
the $n$-torus and the uniform (or Lebesgue) measure on $\operatorname
{int}(P)$, then the map $\alpha\dvtx\operatorname{int}(P) \times T^n
\rightarrow M$ parametrizing a full-measure subset of $M$ in
action-angle coordinates is measure-preserving. In particular, if $f
\dvtx M \rightarrow\mathbb{R}$ is any integrable function then
%
\begin{equation}\label{eqduistermaatheckman}
\hspace*{6pt}\int_M f(x) \, \mathrm{d}m= \int_{P \times T^n}
f(d_1,\ldots ,d_n,\theta _1,\ldots,
\theta_n) \,\mathrm{dVol}_{\mathbb{R}^n} \wedge \mathrm{d}
\theta_1 \wedge\cdots \wedge \mathrm{d}\theta_n
\end{equation}
and if $f(d_1,\ldots,d_n,\theta_1,\ldots,\theta_n) = f_d(d_1,\ldots
,d_n) f_\theta(\theta_1,\ldots,\theta_n)$ then
%
\begin{equation}\label{eqduistermaatheckmanproduct}
\quad\int_M f(x) \, \mathrm{d}m= \int_P
f_d(d_1,\ldots,d_n) \operatorname
{dVol}_{\mathbb{R}^n} \int_{T^n} f_\theta(
\theta_1,\ldots,\theta_n) \,\mathrm{d}\theta_1
\wedge\cdots \wedge \mathrm{d}\theta_n.\hspace*{-6pt}
\end{equation}
\end{theorem}

All this seems forbiddingly abstract, so we give a specific example
which will prove important below. The 2-sphere is a symplectic manifold
where the symplectic form $\omega$ is the ordinary area form, and the
symplectic volume and the Riemannian volume are the same. Any
area-preserving map of the sphere to itself is a symplectomorphism, but
we are interested in the action of the circle on the sphere given by
rotation around the $z$-axis. This action is by area-preserving maps,
and hence by symplectomorphisms, and in fact it is Hamiltonian: the $j$
map pairs the velocity field with the differential of the function $\mu
(x,y,z) = z$, which is the moment map.

We can see that the action preserves the fibers of $\mu$, which are
just horizontal circles on the sphere. Since the dimension of the torus
(1) is half the dimension of the sphere (2), the sphere is then a toric
symplectic manifold. The fixed points of~the torus action are the north
and south poles. The images of these points under the moment map are
the values $+1$ and $-1$, so we expect the moment polytope to be the
convex hull of these points: the interval $[-1,1]$. This is indeed the
image of $\mu(x,y,z) = z$. And, as the Duistermaat--Heckman theorem
claims, \emph{the pushforward of Lebesgue measure on the sphere to
this interval is a constant multiple of the Lebesgue measure on the
line}. This, of course, is exactly Archimedes' theorem, but restated in
a very sophisticated form.

In particular, it means that one can sample points on the sphere
uniformly by choosing their $z$ and $\theta$ coordinates independently
from uniform distributions on the interval and the circle. The
Duistermaat--Heckman theorem extends a similar sampling strategy to any
toric symplectic manifold. The best way to view this sampling strategy,
we think, is as a useful technique in the theory of intrinsic
statistics on Riemannian manifolds (cf. \cite{Pennec2006tx}) which
applies to a special class of manifolds. In principle, one can sample
the entirety of any Riemannian manifold by choosing charts for the
manifold explicitly and then sampling appropriate measures on a
randomly chosen chart. Since the charts are maps from balls in
Euclidean space to the manifold, this reduces the problem to sampling a
ball in $\mathbb{R}^n$ with an appropriate measure. Of course, this
point of
view is so general as to be basically useless in practice: you rarely
have explicit charts for a nontrivial manifold, and the resulting
measures on Euclidean space could be very exotic and difficult to
sample accurately.

Action-angle coordinates, however, give a single ``chart'' with a
simple measure to sample: the product of Lebesgue measure on the convex
moment polytope and the uniform measure on the torus. There is a small
price to pay here. We cannot sample \emph{all} of the toric symplectic
manifold this way. The boundary of $P$ corresponds to a sort of
skeleton inside the toric symplectic manifold $M$, and we cannot sample
this skeleton in any very simple way using action-angle coordinates. Of
course, if we are using the Riemannian (or symplectic) volume of $M$ to
define the probability measure, this is a measure zero subset, so it is
irrelevant to theorems in probability. The benefit is that by deleting
this skeleton, we remove most of the topology of $M$, leaving us with
the topologically very simple sample space $P \times T^{n-3}$.

\section{\texorpdfstring{Toric symplectic structure on random walks or polygonal ``arms.''}{Toric symplectic structure on random walks or polygonal ``arms''}}

We now consider the classical space of random walks of fixed step
length in $\mathbb{R}^3$ and show that the arguments underlying the historical
application of Archimedes' theorem (e.g., in Rayleigh \cite{Rayleigh1919do}) can be viewed as arguments about action-angle
coordinates on this space as a toric symplectic manifold. We denote the
space of open ``arm'' polygons with $n$ edges of lengths $\vec{r} =
(r_1, \ldots, r_n)$ in $\mathbb{R}^3$ by $\operatorname{Arm}_3(n;\vec
{r})$. In particular, the space of
equilateral $n$-edge arms (with unit edges) is denoted $\operatorname
{Arm}_3(n;\vec{1})$. If we
consider polygons related by a translation to be equivalent, the space
$\operatorname{Arm}_3(n;\vec{r})$ is a product $S^2(r_1) \times
\cdots\times S^2(r_n)$ of round
2-spheres with radii given by the $r_i$. The standard probability
measure on this space is the product measure on these spheres; this
corresponds to choosing $n$ independent points distributed according to
the uniform measure on $S^2$ to be the edge vectors of the polygon.

\begin{proposition}
\label{propsymplecticarms}
The space of fixed edgelength open polygonal ``arms'' $\operatorname
{Arm}_3(n;\vec{r})$ is the
product of $n$ round spheres of radii $\vec{r} = (r_1, \ldots, r_n)$.
This is a $2n$-dimensional toric symplectic manifold where the
Hamiltonian torus action is given by rotating each sphere about the
$z$-axis, and the symplectic volume is the standard measure. The moment
map $\mu\dvtx\operatorname{Arm}_3(n;\vec{r})\rightarrow\mathbb
{R}^n$ is given by the \mbox{$z$-}coordinate
of each edge vector, and the image of this map (the moment polytope) is
the hyperbox $\prod_{i=1}^n [-r_i,r_i]$. There is a measure-preserving map
\[
\alpha\dvtx \prod_{i=1}^n [-r_i,r_i]
\times T^n \rightarrow\operatorname {Arm}_3(n;\vec{r})
\]
given explicitly by $\vec{e}_i = (\cos\theta_i \sqrt{1 - z_i^2},
\sin\theta_i \sqrt{r_i^2 - z_i^2}, z_i)$.
\end{proposition}
\begin{pf}
As we mentioned above, the moment polytope is the convex hull of the
images of the fixed points of the Hamiltonian torus action. The only
polygonal arms fixed by the torus action are those where every edge is
in the $\pm z$-direction, so the $z$-coordinates of the fixed points
are indeed the vertices of the hyperbox $\prod_{i=1}^n [-r_i,r_i]$ and
the hyperbox
itself is clearly their convex hull. The $z$-coordinates $z_1, \ldots,
z_n$ and rotation angles $\theta_1, \ldots, \theta_n$ are the
action-angle coordinates on $\operatorname{Arm}_3(n;\vec{r})$ and the
fact that $\alpha$ is
measure-preserving is an immediate consequence of Theorem~\ref{theoremmp}.
\end{pf}

Since we can sample $\prod_{i=1}^n [-r_i,r_i]\times T^n$ directly, this
gives a direct
sampling algorithm for (a full-measure subset of) $\operatorname
{Arm}_3(n;\vec{r})$. Of course,
direct sampling of fixed-edgelength arms is straightforward even
without symplectic geometry, but this description of arm space has
additional implications for confinement problems: if we can describe a
confinement model by additional linear constraints on the action
variables, this automatically yields a toric symplectic structure on
the space of confined arms. We give examples in the next two sections,
then in Section~\ref{subsecsymplecticRayleigh} we use this machinery
to provide a symplectic explanation for Rayleigh's formula for the
probability density function (p.d.f.) of the distance between the
endpoints of a random equilateral arm.

\subsection{\texorpdfstring{Slab-confined arms.}{Slab-confined arms}}
\label{subsecmpsforarms}

One system of linear constraints on the action variables of equilateral
arms is the ``slab'' confinement model.

\begin{definition}\label{defslabarm}
Given a polygon $p$ in $\mathbb{R}^3$ with vertices $v_1, \ldots,
v_n$, let
$\operatorname{zWidth}(p)$ be the maximum absolute value of the
difference between
$z$-coordinates of any two vertices. We define the subspace
$\operatorname{SlabArm}
(n,h) \subset\operatorname{Arm}_3(n;\vec{1})$ to be the space of
equilateral (open) space
$n$-gons up to translation which obey the constraint $\operatorname
{zWidth}(p) \leq h$.
\end{definition}

This is a slab constraint model where the endpoints of the walk are
free (one could also have a model where one or both endpoints are on
the walls of the slab). We now rephrase this slab constraint in
action-angle variables.

\begin{proposition}\label{propslabarmpolytope}
A polygon $p$ in $\operatorname{Arm}_3(n;\vec{1})$ given by
$(z_1,\ldots,z_n,\theta_1,\ldots
,\theta_n)$ in action-angle coordinates lies in the space $\operatorname
{SlabArm}(n,h)$ if
and only if the vector $\vec{z} = (z_1,\ldots,z_n)$ of action
variables lies in the parallelotope $P(n,h)$ given by the collection of
inequalities\vspace*{-6pt}
\[
-1 \leq z_i \leq1,   \qquad -h \leq\sum_{k=i}^j
z_k \leq h
\]
for each $1 \leq i \leq j \leq n$. Hence, there is a measure-preserving map
\[
\alpha\dvtx P(n,h) \times T^n \rightarrow\operatorname{SlabArm}(n,h)
\]
given by restricting the action-angle map of Proposition~\ref{propsymplecticarms}.
\end{proposition}

\begin{pf}
This follows directly from Definition~\ref{defslabarm}: $\sum_{k=i}^j z_k$ is the difference in $z$-height between vertex $i$ and
$j$ so this family of linear constraints encodes $\operatorname
{zWidth}(p) \leq h$.
The other constraints just restate the condition that $\vec{z}$ lies
in the moment polytope $[-1,1]^n$ for $\operatorname{Arm}_3(n;\vec{1})$.
\end{pf}

\begin{corollary}\label{corslabArm}
The probability that $p \in\operatorname{Arm}_3(n;\vec{1})$ lies in
$\operatorname{SlabArm}(n,h)$ is given
by $\operatorname{Vol}P(n,h)/2^n$.
\end{corollary}

This probability function should be useful in computing the entropic
force exerted by an ideal polymer on the walls of a confining slab.
Figure~\ref{figthreeedgeslabpolytopes} shows a collection of these
moment polytopes for different slab widths, and the corresponding volumes.

\begin{figure}

\includegraphics{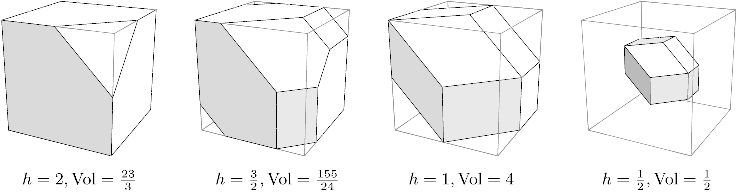}

\caption{This figure shows the moment polytopes corresponding to
$3$-edge arms contained in slabs of width $h$ as subpolytopes of the
cube with vertices $(\pm1, \pm1, \pm1)$, which is the moment
polytope for unconfined arms. In this case, we can compute the volume
of these moment polytopes directly using \emph{polymake}
\cite{Gawrilow2000vl}. We conclude, for instance, that the
probability that
a random $3$-edge arm is confined in a slab of width $\frac{1}{2}$
is $\frac{1}{16}$.}
\label{figthreeedgeslabpolytopes}
\end{figure}

\subsection{\texorpdfstring{Half-space confined arms.}{Half-space confined arms}}

A similar problem is this: suppose we have a freely jointed chain which
is attached at one end to a plane (which we assume for simplicity is
the $xy$-plane), and must remain in the half-space on one side of the
plane. This models a polymer where one end of the molecule is bound to
a surface (at an unknown site). The moment polytope is
%
\begin{equation}\label{eqhalfspacepolytope}
\quad\mathcal{H}_n = \bigl\{\vec{z} \in[-1,1]^n  |
z_1 \geq0, z_1 + z_2 \geq0, \ldots,
z_1 + \cdots+ z_n \geq0, -1 \leq z_i \leq1
\bigr\}\hspace*{-12pt}
\end{equation}
and the analogue of Proposition~\ref{propslabarmpolytope} holds in
this case.

We can understand this condition on arms in terms of a standard random
walk problem: the $z_i$ are i.i.d. steps in a random walk, each
selected from the uniform distribution on $[-1,1]$, and we are
interested in conditioning on the event that all the partial sums are
in $[0,\infty)$. A good deal is known about this problem: for
instance, Caravenna gives an asymptotic p.d.f. for the end of a random
walk conditioned to stay positive, which is the height of the free end of
the chain above the plane~\cite{Caravenna2005ja}. If we could find an
explicit form for this p.d.f., we could analyze the stretching experiment
where the free end of the polymer is raised to a known height above the
plane using magnetic or optical tweezers (cf. \cite{Strick2000vk}).

We can directly compute the partition function for this problem; this
is the volume of subpolytope~(\ref{eqhalfspacepolytope}) of the
hypercube. This result is also stated in a paper of Bernardi,
Duplantier and Nadeau~\cite{Bernardi2010ws}. The proof is a pleasant
combinatorial argument which is tangential to the rest of the paper, so
we relegate it to Appendix~\ref{sechalfSpaceArms}.

\begin{proposition}\label{666666666666}
The volume of the polytope~(\ref{eqhalfspacepolytope}) is $\frac
{1}{2^n} {2n \choose n} = \frac{(2n-1)!!}{n!}$.
\end{proposition}

\subsection{\texorpdfstring{Distribution of failure to close lengths.}{Distribution of failure to close lengths}}
\label{subsecsymplecticRayleigh}

We now apply the action-angle coordinates to give an alternate formula
for the p.d.f. of end-to-end distance in a random walk in $\mathbb{R}^3$ with
fixed step lengths and show that it is equivalent to Rayleigh's
$\operatorname{sinc}$ integral formula~\cite{Rayleigh1919do}. This p.d.f. is key to
determining the Green's function for closed polygons, which in turn is
fundamental to the Moore--Grosberg~\cite{Moore2005fh} and
Diao--Ernst--Montemayor--Ziegler~\cite{Diao2011ie,Diaowt,Diao2012dza} sampling algorithms and to expected
total curvature calculations~\cite{GrosbergExtra,cgks}. For
mathematicians, we note that this p.d.f. is required in order to estimate
the entropic elastic force exerted by an ideal polymer whose ends are
held at a fixed distance. Such experiments are actually done in
biophysics---Wuite et al.~\cite{Wuite2000uu} (cf. \cite
{Bustamante2003tg}) made one of the first measurements of the
elasticity of DNA by stretching a strand of DNA between a bead held in
a micropipette and a bead held in an optical trap.

We first establish some lemmas.

\begin{lemma}
\label{lemsumpdf}
The p.d.f. of a sum of independent uniform random variates in $[-r_1,r_1]$
to $[-r_n,r_n]$ is given by the pushforward of Lesbegue measure on
$\prod_{i=1}^n [-r_i,r_i]$ to $[-\sum r_i,\sum r_i]$ by the linear
function $\sum x_i$.
This p.d.f. is given by
%
\begin{equation}
f_n(x) = \frac{1}{\prod_{i=1}^n 2 r_i} \frac{1}{\sqrt{n}} \operatorname{SA}
(x,r_1,\ldots,r_n),
\end{equation}
where $\operatorname{SA}(x,r_1,\ldots,r_n)$ is the volume of the slice
of the
hypercube $\prod_{i=1}^n [-r_i,r_i]$ by the plane $\sum x_i = x$. The
function $f_n$ is
everywhere $n-2$ times differentiable for $n > 2$.
\end{lemma}

\begin{pf}
It is standard that $f_n$ is a convolution of the $n$ boxcar functions
giving the p.d.f.s of uniform random variates on the intervals
$[-r_1,r_1], \ldots,\break  [-r_n,r_n]$, and hence that~$f_n$ is $n-2$ times
differentiable. The set of points $(x_1,\ldots,x_n)$ with $\sum x_i =
x$ is the slice of the hypercube with $(n-1)$-dimen\-sional volume
$\operatorname{SA}
(x,r_1,\ldots,r_n)$. This not quite the value of the p.d.f. $f_n(x)$, as
we must correct for the rate at which these slices sweep out
$n$-dimensional volume using the coarea formula and normalize the
result by the volume of the hyperbox~$\prod_{i=1}^n [-r_i,r_i]$.
\end{pf}

We have the following.

\begin{proposition}
\label{propend-to-end}
The p.d.f. of the end-to-end distance $\ell\in[0,\sum r_i]$ over the
space of polygonal arms $\operatorname{Arm}_3(n;\vec{r})$ is given by
\[
\phi_n(\ell) = \frac{\ell}{2^{n-1} R \sqrt{n-1}} \bigl( \operatorname{SA}(\ell
-r_n,r_1,\ldots,r_{n-1}) - \operatorname{SA}(
\ell+r_n,r_1,\ldots ,r_{n-1}) \bigr),
\]
%
where $R = \prod_{i=1}^n r_i$ is the product of the edgelengths and
$\operatorname{SA}
(x,r_1,\ldots,r_{n-1})$ is the volume of the slice of the hyperbox
$\prod_{i=1}^{n-1} [-r_i,r_i]$ by the plane $\sum_{i=1}^{n-1} x_i = x$.
\end{proposition}
\begin{pf}
From our moment polytope picture, we can see immediately that the sum
$z$ of the $z$-coordinates of the edges of a random polygonal arm in
$\operatorname{Arm}_3(n;\vec{r})$ has the p.d.f. of a sum of uniform
random variates in $[-r_1,r_1]
\times\cdots\times[-r_n,r_n]$, or $f_n(z)$ in the notation of
Lemma~\ref{lemsumpdf}. Since this is a projection of the spherically
symmetric distribution of end-to-end displacement in $\mathbb{R}^3$ to the
$z$-axis ($\mathbb{R}^1$), equation~(29) of~\cite{Lord1954wh}
applies,\footnote{Lord's notation can be slightly confusing: in his
formula for $p_3(r)$ in terms of $p_1(r)$, we have to remember that
$p_3(r)$ is not itself a p.d.f. on the line, it is a p.d.f. on $\mathbb
{R}^3$. It
only becomes a p.d.f. on the line when multiplied by the correction factor
$4\pi r^2$ giving the area of the sphere at radius $r$ in $\mathbb{R}^3$.}
and tells us that the p.d.f. of $\ell$ is given by
\[
\phi_n(\ell) = -2 \ell f_n'(\ell).
\]

To differentiate $f_n(\ell)$, we use the following observation (cf.
Buonacore \cite{Buonocore2009fk}):
%
\begin{equation}
f_n(x) = \int_{-r_n}^{r_n}
f_{n-1}(x - y) \frac{1}{2r_n} \, \mathrm {d}y= \frac
{F_{n-1}(x + r_n) - F_{n-1}(x - r_n)}{2 r_n},
\end{equation}
where $F_{n-1}(x)$ is the c.d.f. of a sum of uniform random variates in
$[-r_1,r_1],\break  \ldots, [-r_{n-1},r_{n-1}]$. Differentiating and
substituting in the results of Lemma~\ref{lemsumpdf} yields the
formula above.
\end{pf}
Since we will often be interested in equilateral polygons with
edgelength 1, we observe
the following.

\begin{corollary}
\label{corend-to-end}
The\vspace*{1pt} p.d.f. of the end-to-end distance $\ell\in[0,n]$ over the space of
equilateral arms $\operatorname{Arm}_3(n;\vec{1})$ is given by
%
\begin{equation}\label{eqphiell}
\phi_n(\ell) = \frac{\ell}{2^{n-1} \sqrt{n-1}} \bigl( \operatorname{SA}\bigl(\ell
-1,[-1,1]^{n-1}\bigr) - \operatorname{SA}\bigl(\ell+1,[-1,1]^{n-1}
\bigr) \bigr),\hspace*{-6pt}
\end{equation}
%
where $\operatorname{SA}(x,[-1,1]^{n-1})$ is the volume of the slice
of the standard
hypercube $[-1,1]^{n-1}$ by the plane $\sum_{i=1}^{n-1} x_i = x$.
\end{corollary}

The reader who is familiar with the theory of random walks may find the
above corollary rather curious. As mentioned above, the standard
formula for this p.d.f. as an integral of $\operatorname{sinc}$ functions
was given by
Rayleigh in 1919 and it looks nothing like~(\ref{eqphiell}). The
derivation given by Rayleigh of the $\operatorname{sinc}$ integral
formula has no
obvious connection to polyhedral volumes, but in fact by the time of
Rayleigh's paper a connection between polyhedra and $\operatorname
{sinc}$ integrals
had already been given by George P\'olya in his thesis~\cite
{polyathesis,polya} in 1912. This formula has been rediscovered many
times~\cite{Borwein2001bw,Marichal2006tg}. First, we state the
Rayleigh formula~\cite{Rayleigh1919do,Diaowt} in our notation:
%
\begin{equation}\label{eqrayleighform}
\phi_n(\ell) = \frac{2 \ell}{\pi} \int_0^\infty
y \sin\ell y \operatorname{sinc}^n y \, \mathrm{d}y,
\end{equation}
where $\operatorname{sinc}x = \sin x/x$ as usual. Now P\'olya showed
that the volume
of the central slab of the hypercube $[-1,1]^{n-1}$ given by $-a_0 \leq
\sum x_i \leq a_0$ is given by
%
\begin{equation}\label{eqpolyasincintegral}
\operatorname{Vol}(a_0) = \frac{2^n a_0}{\pi} \int
_0^\infty \operatorname{sinc}a_0 y
\operatorname{sinc} ^{n-1} y \, \mathrm{d}y.
\end{equation}
Our $\operatorname{SA}(x,[-1,1]^{n-1})$ is the $(n-1)$-dimensional
volume of a face
of this slab; since it is this face (and its symmetric copy) which
sweep out $n$-dimensional volume as $a_0$ increases, we can deduce that
\[
\operatorname{SA}\bigl(x,[-1,1]^{n-1}\bigr) = \frac{\sqrt{n-1}}{2}
\operatorname{Vol}'(x),
\]
and we can obtain a formula for $\operatorname{SA}(x,[-1,1]^{n-1})$ by
differentiating~(\ref{eqpolyasincintegral}). After some
simplifications, we get
\[
\operatorname{SA}\bigl(x,[-1,1]^{n-1}\bigr) = \frac{2^{n-1} \sqrt{n-1}}{\pi} \int
_0^\infty \cos(xy) \operatorname{sinc}^{n-1}
y \, \mathrm{d}y.
\]
Using the angle addition formula for $\cos(a+b)$, this implies that
\begin{eqnarray*}
&& \operatorname{SA}\bigl(\ell-1,[-1,1]^{n-1}\bigr) - \operatorname{SA}
\bigl(\ell +1,[-1,1]^{n-1}\bigr)\\
 &&\qquad= \frac
{2^{n-1} \sqrt{n-1}}{\pi} \int
_0^\infty2 \sin y \sin\ell y \operatorname{sinc}
^{n-1} y \, \mathrm{d}y
\\
&&\qquad= \frac{2^{n} \sqrt{n-1}}{\pi} \int_0^\infty y \sin\ell y
\operatorname{sinc} ^{n} y \, \mathrm{d}y.
\end{eqnarray*}
Multiplying by $\frac{\ell}{2^{n-1} \sqrt{n-1}}$ shows that~(\ref
{eqphiell}) and~(\ref{eqrayleighform}) are equivalent formulas for
the p.d.f. $\phi_n$.

Given~(\ref{eqphiell}) and~(\ref{eqrayleighform}), the p.d.f. of the
failure-to-close vector $\vec{\ell} = \sum\vec{e}_i$ with length
$|\vec{\ell} | = \ell$ can be written in the following forms:
%
\begin{eqnarray}
\nonumber
\Phi_n(\vec{\ell} ) &=& \frac{1}{4\pi\ell^2}
\phi_n(\ell )
\\
\label{eqpdffailure-to-closevector}
&=&  \frac{1}{2^{n+1} \pi\ell\sqrt{n-1}} \bigl( \operatorname {SA}\bigl(
\ell -1,[-1,1]^{n-1}\bigr) - \operatorname{SA}\bigl(\ell+1,[-1,1]^{n-1}
\bigr) \bigr)
\\
\nonumber
 & =& \frac{1}{2\pi^2 \ell} \int_0^\infty
y \sin\ell y \operatorname{sinc}^n y \, \mathrm{d}y.
\end{eqnarray}
The latter formula for the p.d.f. appears in Grosberg and Moore~\cite
{Moore2005fh} as equation~(B5). Since Grosberg and Moore then actually
evaluate the integral for the p.d.f. as a finite sum, one immediately
suspects that there is a similar sum form for the slice volume terms
in~(\ref{eqphiell}). In fact, we have several options to choose
from, including using P\'olya's finite sum form to express~(\ref
{eqpolyasincintegral}) and then differentiating the sum formula with
respect to the width of the slab. We instead rely on the following
theorem, which we have translated to the current situation.

\begin{theorem}[(Marichal and Mossinghoff~\cite{Marichal2006tg})]
Suppose that $\vec{w} \in\mathbb{R}^n$ has all nonzero components and
suppose $x$ is any real number. Then the $(n-1)$-dimensional volume of
the intersection of the hyperplane $\anglel  \vec{x}, \vec{w} \angler
= x$ with the hypercube $[-1,1]^n$ is given by
%
\begin{equation}\label{eqmarichalsliceformula}
\operatorname{Vol}= \frac{|\vec{w}|_2}{(n-1)!   \prod w_i} \sum_{A
\subset\{
1,\ldots,n\}}
(-1)^{|A|} \biggl(x + \sum_{i \notin A}
w_i - \sum_{i
\in A} w_i
\biggr)_+^{n-1},
\end{equation}
where $|\vec{w}|_2$ is the usual ($L^2$) norm of the vector $\vec
{w}$, $z_+ = \max(z,0)$ and we use the convention $0^0 = 0$ when
considering the $n=1$ case.
\end{theorem}

For our $\operatorname{SA}(x,[-1,1]^{n-1})$ function, the vector $\vec{w}$ consists of all 1's.
Using the fact that the number of subsets of $\{1,\dots,n\}$ with cardinality $k$
is ${n \choose k}$, we can prove the following proposition.

\begin{proposition}\label{propsaexplicitsum}
The $(n-2)$-dimensional volume $\operatorname{SA}(x,[-1,1]^{n-1})$ is
given by
%
\begin{equation}\label{eqfinitesumformula}
\quad\operatorname{SA}\bigl(x,[-1,1]^{n-1}\bigr) = \frac{\sqrt{n-1}}{(n-2)!} \sum
_{k=0}^{n-1} (-1)^k  \pmatrix{n-1 \cr
k} (x + n-1 - 2k)_+^{n-2}.\hspace*{-6pt}
\end{equation}
\end{proposition}

We can combine this with~(\ref{eqpdffailure-to-closevector}) to
obtain the explicit piecewise polynomial p.d.f. for the failure-to-close
vector (for $n \geq2$):
%
\begin{eqnarray}
\quad\Phi_n(\vec{\ell})&=& \frac{n-1}{2^{n+1}\pi\ell}
\nonumber
\\[-8pt]
\label{eqftcpdfv2}
\\[-8pt]
\nonumber
&&{}\times\sum
_{k=0}^{n-1}\frac{(-1)^k }{k!(n-k-1)!} \bigl((n+\ell
-2k-2)_+^{n-2}-(n+\ell-2k)_+^{n-2} \bigr).
\end{eqnarray}
When $n=2$, recall that we use the convention $0^0 = 0$. When $n=1$ the
formula does not make sense, but we can easily compute $\Phi_1(\vec
{\ell}) = \frac{1}{4\pi}\delta(1 - \ell)$. This formula for $\Phi
_n(\ell)$ is known classically, and given as (2.181) in Hughes~\cite
{hughes1995random}. The polynomials are precisely those given in (B13)
of Moore and Grosberg~\cite{Moore2005fh}.

%
%

\subsection{\texorpdfstring{The expected total curvature of equilateral polygons.}{The expected total curvature of equilateral polygons}}
\label{subsectotalcurvatureofclosedpolygons}

In Section~\ref{secnumerics}, it will be useful to know exact values
of the expected total curvature of equilateral polygons. Let
$\operatorname{Pol}_3(n;\vec{1})
\subset\operatorname{Arm}_3(n;\vec{1})$ be the subspace of closed
equilateral $n$-gons.
Following the approach of~\cite{GrosbergExtra,cgks}, we can use the
p.d.f. above to find an integral formula for the expected total curvature
of an element of $\operatorname{Pol}_3(n;\vec{1})$:

\begin{theorem}\label{thmexpectedtotalcurvature}
The expected total curvature of an equilateral $n$-gon is
%
\begin{equation}
\label{eqtotalcurvature}
E\bigl(\kappa; \operatorname{Pol}_3(n;\vec{1})\bigr) =
\frac{n}{2C_n}\int_0^2 \arccos \biggl(
\frac
{\ell^2-2}{2} \biggr) \Phi_{n-2}(\ell)\ell\,\mathrm{d}\ell,
\end{equation}
where $C_n$ and $\Phi_{n-2}(\ell)$ are given explicitly in~(\ref
{eqcn}) and~(\ref{eqftcpdfv2}), respectively, and Table~\ref{tabexpectedtotalcurvature} shows exact values of the integral for
small $n$.
\end{theorem}

%
This integral can be evaluated easily by computer algebra since $\Phi_{n-2}(\ell)$ is piecewise polynomial in $\ell$ and since
${\int_0^2 \arccos (\frac{\ell^2-2}{2}) \ell^k \,\mathrm{d}\ell = \frac{2^{2k+1}n \mathrm{B}({k}/{2}+1,{k}/{2})}{(k+1)^2}}$,
where $\mathrm{B}$ is the Euler beta function.
Of course, it would be very interesting to find a closed\vadjust{\goodbreak} form.

\begin{pf*}{Proof of Theorem~\protect\ref{thmexpectedtotalcurvature}}
The total curvature of a polygon is just the sum of the turning angles,
so the expected total curvature of an $n$-gon is simply $n$ times the
expected value of the turning angle $\theta(\vec{e}_i, \vec
{e}_{i+1})$ between any pair $(\vec{e}_i,\vec{e}_{i+1})$ of
consecutive edges. In other words,
%
\begin{eqnarray}
E\bigl(\kappa; \operatorname{Pol}_3(n;\vec{1})\bigr) &=&
n E\bigl(\theta ;\operatorname{Pol}_3(n;\vec{1})\bigr)
\nonumber
\\[-8pt]
\label{eqexpectedtotalcurvaturegeneral}
\\[-8pt]
\nonumber
&=& n \int
\theta(\vec {e}_i,\vec{e}_{i+1}) P(\vec{e}_i,
\vec{e}_{i+1})\, \mathrm{dVol}_{\vec{e}_i}\,\mathrm{dVol}_{\vec{e}_{i+1}},
\end{eqnarray}
where $P(\vec{e}_i,\vec{e}_{i+1}) \,\mathrm{dVol}_{\vec
{e}_i}\,\mathrm{dVol}_{\vec
{e}_{i+1}}$ is the joint distribution of the pair of edges.

The edges $\vec{e}_i, \vec{e}_{i+1}$ are chosen uniformly from the
unit sphere subject to the constraint that the remaining $n-2$ edges
must connect the head of $\vec{e}_{i+1}$ to the tail of $\vec{e}_i$.
In other words,
\begin{eqnarray*}
&& P(\vec{e}_i, \vec{e}_{i+1}) \,\mathrm{dVol}_{\vec
{e}_i}
\,\mathrm{dVol}_{\vec
{e}_{i+1}} \\
&&\qquad= \frac{1}{C_n} \Phi_1(
\vec{e}_i) \Phi_1(\vec{e}_{i+1})
\Phi_{n-2}(-\vec{e}_i-\vec{e}_{i+1})
\,\mathrm{dVol}_{\vec
{e}_i} \,\mathrm{dVol} _{\vec{e}_{i+1}},
\end{eqnarray*}
where
%
\begin{equation}
\label{eqcn} C_n = \Phi_n(\vec{0}) = \frac{1}{2^{n+1}\pi(n-3)!}
\sum_{k=0}^{\lfloor{n}/{2}\rfloor} (-1)^{k+1} \pmatrix{n \cr k}(n-2k)^{n-3}
\end{equation}
is the normalized $(2n-3)$-dimensional Hausdorff measure of the
submanifold of closed $n$-gons. Notice that $\Phi_1(\vec{v}) = \frac
{\delta(|\vec{v}|-1)}{4\pi}$ is the distribution of a point chosen
uniformly on the unit sphere. In particular, we can rewrite the
integral~(\ref{eqexpectedtotalcurvaturegeneral}) as
\begin{eqnarray*}
&& E\bigl(\kappa; \operatorname{Pol}_3(n;\vec{1})\bigr)\\
&&\qquad =
\frac{n}{C_n} \int_{\vec{e}_i\in S^2} \!\int_{\vec{e}_{i+1} \in S^2}
\theta(\vec{e}_i,\vec{e}_{i+1}) \frac
{1}{16\pi^2}
\Phi_{n-2}(-\vec{e}_i-\vec{e}_{i+1})
\,\mathrm{dVol}_{S^2} \,\mathrm{dVol}_{S^2}.
\end{eqnarray*}
Moreover, at the cost of a constant factor $4\pi$ we can integrate out
the $\vec{e}_i$ coordinate and assume $\vec{e}_i$ points in the
direction of the north pole. Similarly, at the cost of an additional
$2\pi$ factor we can integrate out the azimuth angle of $\vec
{e}_{i+1}$ and reduce the above integral to a single integral over the
polar angle of $\vec{e}_{i+1}$, which is now exactly the angle $\theta
(\vec{e}_i,\vec{e}_{i+1})$:
\[
E\bigl(\kappa; \operatorname{Pol}_3(n;\vec{1})\bigr) =
\frac{n}{2C_n} \int_0^\pi\theta  \Phi
_{n-2}(\sqrt{2-2\cos\theta}) \sin\theta\,\mathrm{d}\theta
\]
since $\sqrt{2-2\cos\theta}$ is the length of the vector $\vec{\ell
} = -\vec{e}_i - \vec{e}_{i+1}$. Changing coordinates to integrate
with respect to $\ell= |\vec{\ell}| \in[0,2]$ completes the proof.
\end{pf*}

\section{\texorpdfstring{The (almost) toric symplectic structure on closed polygons.}{The (almost) toric symplectic structure on closed polygons}}

We are now ready to describe explicitly the toric symplectic structure
on closed polygons of fixed edgelengths. We first need to fix a bit of
notation. The space $\operatorname{Pol}_3(n;\vec{r})$ of closed
polygons of fixed edgelengths
$\vec{r} = (r_1, \ldots, r_n)$, where polygons related by translation
are considered equivalent, is a subspace of the Riemannian manifold
$\operatorname{Arm}_3(n;\vec{r})$ (with the product metric on spheres
of varying radii). It has
a corresponding subspace metric and measure, which we refer to as the
\emph{standard measure} on $\operatorname{Pol}_3(n;\vec{r})$. There
is a measure-preserving
action of $\operatorname{SO}(3)$ on $\operatorname{Pol}_3(n;\vec
{r})$, and a corresponding quotient space
$\widehat{\operatorname{Pol}}_3(n;\vec{r})= \operatorname
{Pol}_3(n;\vec{r})/\operatorname{SO}(3)$. This quotient space
inherits a pushforward
measure from the standard measure on $\operatorname{Pol}_3(n;\vec
{r})$, and we call this the
standard measure on $\widehat{\operatorname{Pol}}_3(n;\vec{r})$,
which we will shortly see (almost) has a
toric symplectic structure.

We can triangulate a convex $n$-gon by joining vertices $v_3, \ldots,
v_{n-1}$ to $v_1$ with $n-3$ chords to create $n-2$ triangles. This
triangulation, which we call the ``fan triangulation,'' is shown in
Figure~\ref{figfantriangulation}. There are many other ways to
triangulate the polygon, but---as can be proved inductively---each
consists of $n-2$ triangles formed by $n-3$ chords.

%
\begin{figure}

\includegraphics{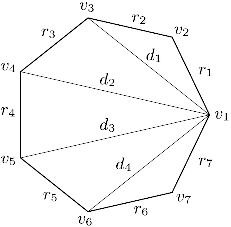}

\caption{The fan triangulation of the regular planar 7-gon.}
\label{figfantriangulation}
\end{figure}

We call these $n-3$ chords the \emph{diagonals} of the triangulation
$T$. Since the side lengths of any triangle obey 3 triangle
inequalities, the edgelengths and diagonal lengths of $T$ must obey a
set of $3(n-2)$ triangle inequalities, which we call the \emph
{triangulation inequalities}. For the fan triangulation, let $r_1,
\ldots, r_n$ be the edgelengths of an $n$-gon and let $d_1, \ldots,
d_{n-3}$ be the lengths of the diagonals. In this triangulation, $d_i =
|v_{i+2} - v_1|$. The first and last triangles are made up of two sides
and one diagonal: $r_1$, $r_2$, and $d_1$, or $r_{n-1}$, $r_n$ and
$d_{n-3}$. So these variables must satisfy the triangle inequalities
%
\begin{equation}\label{eqfirstandlasttriangulationinequality}
\begin{array}{l}
\displaystyle d_1 \leq r_1 +
r_2,
\\[3pt]
\displaystyle r_1 \leq d_1 + r_2,
\\[3pt]
\displaystyle r_2 \leq r_1 + d_1,
\end{array}
 \quad\mbox{and} \quad
\begin{array}{l}
\displaystyle d_{n-3}
\leq r_{n-1} + r_n,
\\[3pt]
\displaystyle r_{n-1} \leq d_{n-3} + r_n,
\\[3pt]
\displaystyle r_n \leq r_{n-1} + d_{n-3}.
\end{array}
\end{equation}
All other triangles are made up of two diagonals and one side: the
triangle $\triangle v_1 v_{i+2} v_{i+3}$ has side lengths $d_i$,
$r_{i+2}$, and $d_{i+1}$. These variables must satisfy the triangle\vspace*{-3pt} inequalities
%
\begin{equation}\label{eqmiddletriangulationinequality}
r_{i+1} \leq d_i + d_{i+1},\qquad
d_i \leq r_{i+2} + d_{i+1},\qquad   d_{i+1}
\leq r_{i+2} + d_i.
\end{equation}

Finally, given a diagonal (chord) of a space polygon, we can perform
what the random polygons community calls a \emph{polygonal fold} or
\emph{crankshaft move}~\cite{Anonymous2010p2603} and the symplectic
geometry community calls a \emph{bending flow}~\cite
{Kapovich1996p2605} by rotating one arc of the polygon rigidly with
respect to the complementary arc, with axis of rotation
the diagonal, as shown in Figure~\ref{figbendingflow}; the collection of such rotations around all of the $n - 3$ diagonals of a given triangulation will be our Hamiltonian torus action.
%
\begin{figure}

\includegraphics{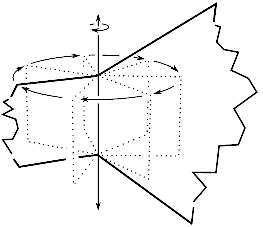}
\vspace*{-3pt}
\caption{In a bending flow or polygonal fold, we use two vertices of
the polygon to define an axis of rotation and rotate one arc of the
polygon (shown at left) around this axis while the complementary arc of
the polygon (shown at right) stays fixed. All edgelengths are fixed by
this transformation and the polygon stays\vspace*{-6pt} closed.}
\label{figbendingflow}
\end{figure}

We can now summarize the existing literature as follows.

\begin{theorem}[(Kapovich and Millson~\cite{Kapovich1996p2605},
Howard, Manon and  Millson~\cite{Howard2008uy}, Hitchin~\cite{Hitchin1987uk})]
\label{thmsymplecticclosedpolygons}
The following facts are known:
\begin{itemize}
\item$\widehat{\operatorname{Pol}}_3(n;\vec{r})$ is a possibly
singular $(2n-6)$-dimensional symplectic
manifold. The symplectic volume is equal to the standard measure.
\item To any triangulation $T$ of the standard $n$-gon we can associate
a Hamiltonian action of the torus $T^{n-3}$ on $\widehat{\operatorname
{Pol}}_3(n;\vec{r})$, where the angle
$\theta_i$ acts by folding the polygon around the $i$th diagonal of
the triangulation.
\item The moment map $\mu\dvtx \widehat{\operatorname{Pol}}_3(n;\vec
{r})\rightarrow\mathbb{R}^{n-3}$ for a
triangulation $T$ records the lengths $d_i$ of the $n-3$ diagonals of
the triangulation.
\item The moment polytope $P$ is defined by the triangulation
inequalities for $T$.
%
\item The action-angle map $\alpha$ for a triangulation $T$ is given by constructing the triangles using
the diagonal and edgelength data to recover their side lengths, and assembling them in space
with (oriented) dihedral angles given by the $\theta_i$, as shown in Figure~\ref{figactionmap}.
\item The inverse image $\mu^{-1}(\operatorname{interior} P) \subset
\widehat{\operatorname{Pol}}_3(n;\vec{r})$ of the interior of the
moment polytope $P$ is an (open) toric
symplectic\vspace*{-3pt} manifold.
\end{itemize}
\end{theorem}

Here is a very brief summary of how these results work. Just as for
%
\begin{figure}[t]

\includegraphics{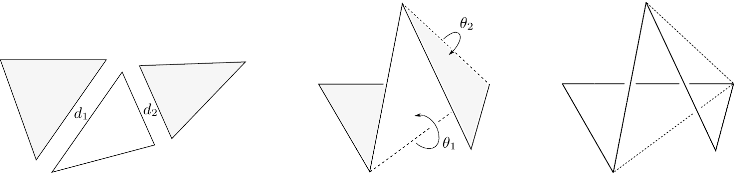}

\caption{This figure shows how to construct an equilateral pentagon in
$\widehat{\operatorname{Pol}}(5;\vec{1})$ using the action-angle
map. First, we pick a point in the moment polytope shown in
Figure~\protect\ref
{figfantriangulationpolytopes} at center. We have now specified
diagonals $d_1$ and $d_2$ of the pentagon, so we may build the three
triangles in the triangulation from their side lengths, as in the
picture at left. We then choose dihedral angles $\theta_1$ and $\theta
_2$ independently and uniformly, and join the triangles along the
diagonals $d_1$ and $d_2$, as in the middle picture. The right-hand
picture shows the final space polygon, which is the boundary of this
triangulated surface.}
\label{figactionmap}
\end{figure}
Hamiltonian torus actions, in general there is a moment map associated
to every Hamiltonian Lie group action on a symplectic manifold. In
particular, Kapovich and Millson~\cite{Kapovich1996p2605} pointed out
that the symplectic manifold $\operatorname{Arm}_3(n;\vec{r})$ admits
a Hamiltonian action by
the Lie group $\operatorname{SO}(3)$ given by rotating the polygonal
arm in space
[this is the diagonal $\operatorname{SO}(3)$ action on the product of
spheres]. In
this case, there are three circle actions given by rotating around the
$x$-, $y$- and $z$-axes, each of which defines a conserved quantity.
But these circle actions do not commute: the three quantities conserved
under each rotation are the coordinate functions of a map $\mu\dvtx
\operatorname{Arm}_3(n;\vec{r})
\rightarrow\mathbb{R}^3$ which is \emph{equi}variant under the
$\operatorname{SO}(3)$
action but not \emph{in}variant. In fact, adapting the computation we
did above in our symplectic explanation of Archimedes' theorem, we can
see that $\mu$ is the displacement vector joining the ends of the polygon.

The closed polygons $\operatorname{Pol}_3(n;\vec{r})$ are the fiber
$\mu^{-1}(\vec{0})$ of
this map. This fiber of $\mu$ is preserved by the $\operatorname
{SO}(3)$ action. In
this situation, we can perform what is known as a \emph{symplectic
reduction} (or Marsden--Weinstein--Meyer reduction \mbox{\cite{Marsden1974tg,Meyer1973wu}}, see Part IX of~\cite
{CannasdaSilva2001cg}) to produce a symplectic structure on the
quotient of the fiber $\mu^{-1}(\vec{0})$ by the group action. This
yields a symplectic structure on the $(2n-6)$-di\-men\-sional moduli space
$\widehat{\operatorname{Pol}}_3(n;\vec{r})$. The symplectic measure
induced by this symplectic structure is
equal to the standard measure given by pushing forward the Hausdorff
measure on $\operatorname{Pol}_3(n;\vec{r})$ to $\widehat
{\operatorname{Pol}}_3(n;\vec{r})$ because the ``parent'' symplectic
manifold $\operatorname{Arm}_3(n;\vec{r})$ is a K\"ahler
manifold~\cite{Hitchin1987uk}.

The polygon space $\widehat{\operatorname{Pol}}_3(n;\vec{r})$ is
singular if
\[
\varepsilon_I(\vec{r}) := \sum_{i\in I}
r_i - \sum_{j \notin I} r_j
\]
is zero for some $I \subset\{1, \ldots, n\}$. Geometrically, this
means it is possible to construct a degenerate polygon which lies on a
line with edgelengths given by $\vec{r}$. Since these polygons are
fixed by rotations around the line on which they lie, the action of
$\operatorname{SO}(3)$ is not free in this case and the symplectic reduction
develops singularities. Nonetheless, the reduction $\widehat
{\operatorname{Pol}}_3(n;\vec{r})$ is a complex
analytic space with isolated singularities; in particular, the
complement of the singularities is a symplectic (in fact K\"ahler)
manifold to which Theorem~\ref{thmsymplecticclosedpolygons} applies.

Both the volume and the cohomology ring of $\widehat{\operatorname
{Pol}}_3(n;\vec{r})$ are well understood
from this symplectic perspective~\cite
{Brion1991hz,Kirwan1992by,Hausmann1998vx,Kamiyama1999tj,Takakura2001ir,Khoi2005ch,Mandini2008wq}.
For example, we have the following.

\begin{proposition}[(Takakura~\cite{Takakura2001ir}, Khoi~\cite
{Khoi2005ch}, Mandini~\cite{Mandini2008wq})]\label{proppolygonspacevolume}
The volume of $\widehat{\operatorname{Pol}}_3(n;\vec{r})$ is
\[
\operatorname{Vol}\bigl(\widehat{\operatorname{Pol}}_3(n;\vec{r})
\bigr) = -\frac{(2\pi)^{n-3}}{2(n-3)!} \sum_I
(-1)^{n-|I|} \varepsilon_I(\vec{r})^{n-3},
\]
where the sum is over all $I \subset\{1, \ldots, n\}$ such that
$\varepsilon_I(\vec{r}) > 0$.
\end{proposition}

\begin{corollary}\label{corequilateralvolume}
The volume of the space of equilateral $n$-gons is
\[
\operatorname{Vol}\bigl(\widehat{\operatorname{Pol}}_3(n;\vec{1})
\bigr) = -\frac{(2\pi)^{n-3}}{2(n-3)!} \sum_{k=0}^{\lfloor
{n}/{2}\rfloor}
(-1)^{k} \pmatrix{n \cr k} (n-2k)^{n-3}.
\]
\end{corollary}


\subsection{\texorpdfstring{The knotting probability for equilateral hexagons.}{The knotting probability for equilateral hexagons}}

We immediately give an example application of this picture. In~\cite
{cgks}, we showed using the F\'ary--Milnor theorem that at least
$\frac{1}{3}$ of hexagons of total length 2 are unknotted by
showing that their total curvature was too small to form a knot. We
could repeat the calculation using our explicit formula for the
expectation of the total curvature for equilateral hexagons above, but
the results would be disappointing; only about $27\%$ of the space is
revealed to be unknotted by this method. On the other hand action-angle
coordinates, coupled with results of Calvo, immediately yield a better bound.

\begin{proposition}
At least $\frac{1}{2}$ of the space $\widehat{\operatorname
{Pol}}_3(6;\vec{1})$ of equilateral
hexagons consists of unknots.
\end{proposition}

\begin{pf}
There are several triangulations of the hexagon, but only two have a
central triangle surrounded by 3 others: the triangulations $T_{135}$
given by joining vertices $1$--$3$--$5$ and $T_{246}$ given by joining
vertices $2$--$4$--$6$. Each has a corresponding set of action-angle
coordinates $\alpha\dvtx \mathcal{P} \times T^3 \rightarrow\widehat
{\operatorname{Pol}}_3(6;\vec{1})$.
In \cite{Calvo2001cp}, an impressively detailed analysis of hexagon
space, Jorge Calvo defines a geometric\footnote{Interestingly, curl is
independent from the topological invariant given by the handedness of
the trefoil, so there are at least four different types of equilateral
hexagonal trefoils. Calvo proves that curl and handedness together form
a complete set of invariants for equilateral hexagonal trefoils; that
is, there are only four types.} invariant of hexagons called the curl
which is $0$ for unknots and $\pm1$ for trefoils. In the proof of his
Lemma~16, Calvo observes that any knotted equilateral hexagon with curl
${+}1$ has all three dihedral angles between $0$ and $\pi$ in either
$T_{135}$ or $T_{246}$.

The rest of the proof is elementary, but we give all the steps here as
this is the first of many such arguments below. Formally, the knot
probability is the expected value of the characteristic function
\[
\chi_{\mathrm{knot}}(p) =
\cases{ 1,  & \quad\mbox{if $p$ is knotted,}\vspace*{3pt}
\cr
0, & \quad\mbox{if $p$ is unknotted.}}
\]
By Calvo's work, $\chi_{\mathrm{knot}}$ is bounded above by the sum
${\chi_{\mathrm{curl}=+1}+\chi_{\mathrm{curl}=-1}}$ and $\chi_{\mathrm
{curl}=+1}$ is bounded above by the sum of the characteristic functions
\[
\chi_T(d_1,d_2,d_3,
\theta_1,\theta_2,\theta_3) = %
\cases{ 1, & \quad\mbox{if $\theta_i \in[0,\pi]$ for $i \in\{1, 2, 3\}$,}
\vspace*{3pt}
\cr
0, &\quad \mbox{otherwise,}}
\]
where $T$ is either $T_{135}$ or $T_{246}$. Now Theorem~\ref
{thmsymplecticclosedpolygons} tells us that almost all of $\widehat
{\operatorname{Pol}}_3(6;\vec{1})
$ is a toric symplectic manifold, so (\ref{eqduistermaatheckmanproduct}) of Theorem~\ref{theoremmp} holds for integrals over this
polygon space.\vspace*{1.5pt} In particular, $\chi_T$ does not depend on the $d_i$,
so its expected value over $\widehat{\operatorname{Pol}}_3(6;\vec
{1})$ is equal to its expected value
over the torus $T^3$ of $\theta_i$. This expected value is clearly
$\frac{1}{8}$. Summing over both triangulations and making a
similar argument for $\chi_{\mathrm{curl}=-1}$, we see the knot
probability is no more than $\frac{1}{2}$, as desired.
\end{pf}

\begin{figure}

\includegraphics{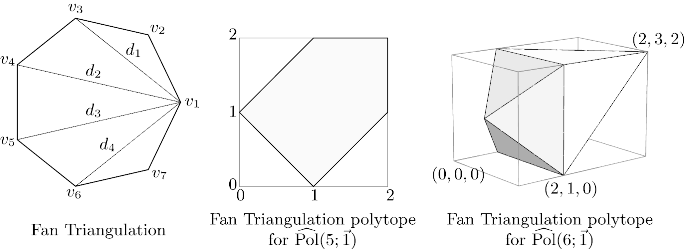}

\caption{This figure shows the fan triangulation of a 7-gon on the
left and the corresponding moment polytopes for equilateral space
pentagons and equilateral space hexagons. For the pentagon moment
polytope, we show the square with corners at $(0,0)$ and $(2,2)$ to
help locate the figure, while for the hexagon moment polytope, we show
the box with corners at $(0,0,0)$ and $(2,3,2)$ to help understand the
geometry of the figure. The vertices of the polytopes correspond to
polygons fixed by the torus action given by rotating around the
diagonals. The polygons on the boundary of the moment polytope all
degenerate in some way, as at least one triangle inequality is
extremized; the vertices of the moment polytope represent especially
degenerate polygons which extremize several triangle inequalities at
once. For instance, the $(2,2)$ point in the pentagon's moment polytope
corresponds to the configuration given by an isoceles triangle with
sides $2$, $2$, and $1$ (two triangles have collapsed to line
segments). The diagonals lie along the long sides; rotating around them
is a rotation of the entire configuration in space, and is hence
trivial because we are considering equivalence classes up to the action
of $\operatorname{SO}(3)$. The $(2,3,2)$ point in the hexagon's moment polytope
corresponds to a completely flat (or ``lined'') configuration
double-covering a line segment of length $3$. Here, all the diagonals
lie along the same line and rotation around the diagonals does nothing.}
\label{figfantriangulationpolytopes}
\end{figure}

Of course, this bound is still a substantial underestimate of the
fraction of unknots. Over a 12-hour run of the ``PTSMCMC'' Markov chain
sampler of Section~\ref{subsecptsmcmc}, we examined 1,318,001
equilateral hexagons and found 173~knots. Using the $95\%$ confidence
level Geyer IPS error estimators of Section~\ref{subsecgeyerips}, we
estimate the knot probability for unconfined equilateral hexagons is \mbox{$1.3 \times10^{-4} \pm0.2 \times10^{-4}$}, or between $1.1$ and
$1.5$ in $10{,}000$.

\subsection{\texorpdfstring{The fan triangulation and chordlengths.}{The fan triangulation and chordlengths}}

As we noted above, the ``fan'' triangulation of a polygon is created by
joining vertex $v_1$ to vertices $v_3, \ldots, v_{n-1}$. Recall that as
shown in Figure~\ref{figfantriangulationpolytopes}, we number the
diagonals $d_1, \ldots, d_{n-3}$ so that the first triangle has
edgelengths $d_1$, $r_1$, $r_2$, the last triangle has edgelengths
$d_{n-3}$, $r_{n-1}$, $r_{n}$, and all the triangles in between have
edgelengths in the form $d_i$, $d_{i+1}$, $r_{i+2}$. The corresponding
triangulation inequalities, which we call the ``fan triangulation
inequalities'' are then
%
\begin{eqnarray}
|r_1 - r_2|  &\leq &  d_1 \leq r_1 +
r_2,\qquad   
%
%
 r_{i+2} \leq d_i + d_{i+1},
\nonumber
\\[-8pt]
\label{eqfanpolytope}
\\[-8pt]
\nonumber
|d_i - d_{i+1}|  &\leq &  r_{i+2}, \qquad
%
%
  |r_n - r_{n-1}| \leq  d_{n-3} \leq
r_n + r_{n-1}.
\end{eqnarray}

\begin{definition}\label{deffanpolytope}
The \emph{fan triangulation polytope} $P_n(\vec{r}) \subset\mathbb{R}
^{n-3}$ is the moment polytope for $\widehat{\operatorname
{Pol}}_3(n;\vec{r})$ corresponding to the fan
triangulation and is determined by the fan triangulation
inequalities~(\ref{eqfanpolytope}). The fan triangulation polytopes
$P_5(\vec{1})$ and $P_6(\vec{1})$ are shown in Figure~\ref{figfantriangulationpolytopes}.
\end{definition}

This description of the moment polytope follows directly from
Theorem~\ref{thmsymplecticclosedpolygons}.




Applying Theorem~\ref{theoremmp} to this situation gives necessary
and sufficient conditions for uniform sampling on $\widehat
{\operatorname{Pol}}_3(n;\vec{r})$. These could
be used to test proposed polygon sampling algorithms given statistical
tests for uniformity on convex subsets of Euclidean space and on the
$(n-3)$-torus.

\begin{proposition}\label{proppolygonsampling}
A polygon in $\widehat{\operatorname{Pol}}_3(n;\vec{r})$ is sampled
according to the standard measure if
and only if its diagonal lengths $d_1 = |v_1 - v_3|$, $d_2 = |v_1 -
v_4|$, \ldots, $d_{n-3} = |v_1 - v_{n-1}|$ are uniformly sampled from
the fan polytope $P_n(\vec{r})$ and its dihedral angles around these
diagonals are sampled independently and uniformly in $[0,2\pi)$.
\end{proposition}

The fan triangulation polytope also gives us a natural way to
understand the probability distribution of chord lengths of a closed
random walk.
To fix notation, we make the following definition.

\begin{definition}
Let $\operatorname{ChordLength}(k,n;\vec{r})$ be the length $|v_1 -
v_{k+1}|$ of the chord skipping\vspace*{1.5pt} the
first $k$ edges in a polygon sampled according to the standard measure
on $\widehat{\operatorname{Pol}}_3(n;\vec{r})$. This is a random variable.
\end{definition}

The expected values of squared chordlengths for equilateral polygons
have been computed by a rearrangement technique, and turn out to be
quite simple.

\begin{proposition}[(Cantarella, Deguchi, Shonkwiler~\cite
{CPACPA21480} and Millett, Zirbel~\cite{Zirbel2012gg})]
\label{propsecondmoment}
The second moment of the random variable $\operatorname
{ChordLength}(k,n;\vec{1})$ is $\frac{k(n-k)}{n-1}$.
\end{proposition}

It is obviously interesting to know the other moments of these random
variables, but this problem seems considerably harder. In particular,
the techniques used in the proofs of~Proposition~\ref{propsecondmoment} do not apply to other moments of chordlength. Here is an
alternate form for the chordlength problem which allows us to make some
explicit calculations.

\begin{theorem}
\label{thmexpectedchordlength}
The expectation of the random variable $\operatorname
{ChordLength}(k,\break n;\vec{1})$ is the coordinate
$d_{k-1}$ of the center of mass of the fan triangulation
polytope~$P_n(\vec{1})$. For $n$ between $4$ and $8$, these
expectations are given by the fractions
%
\begin{equation}
%
\begin{array} {@{}l|c@{\quad}c@{\quad}c@{\quad}c@{\quad}c@{}} 
n \setminus
\raisebox{0.30em}{k\,}
& 2 & 3 & 4 & 5 & 6
\\
\hline
4 & \phantom{000}1\phantom{/000} & & &
\\
5 & \phantom{00}17/15\phantom{0} & \phantom{00}17/15\phantom{0} & &
\\
6 & \phantom{00}14/12\phantom{0} & \phantom{00}15/12\phantom{0} & \phantom{00}14/12\phantom{0} &
\\
7 & \phantom{0}461/385 & \phantom{0}506/385 & \phantom{0}506/385 & \phantom{0}461/385
\\
8 & 1168/960 & 1307/960 & 1344/960 & 1307/960 & 1168/960
\end{array} %
\end{equation}
%
The $p$th moment of $\operatorname{ChordLength}(k,n;\vec{1})$ is
coordinate $d_{k-1}$ of the $p$th
center of mass of the fan triangulation polytope $P_n(\vec{1})$.
\end{theorem}

\begin{pf}
Since the measure on $\widehat{\operatorname{Pol}}_3(n;\vec{1})$ is
invariant under permutations of the
edges, the p.d.f. of chord length for any chord skipping $k$ edges must be
the same as the p.d.f. for the length of the chord joining $v_1$ and
$v_{k+1}$. But this chord is a diagonal of the fan triangulation, so
its length is the coordinate $d_{k-1}$ of the fan triangulation
polytope $P_n(\vec{1})$. Since these chord lengths do not depend on
dihedral angles, their expectations over polygon space are equal to
their expectations over $P_n(\vec{1})$ by (\ref{eqduistermaatheckmanproduct}) of Theorem~\ref{theoremmp}, which applies by
Theorem~\ref{thmsymplecticclosedpolygons}. But the expectation of
the $p$th power of a coordinate over a region is simply a coordinate of
the corresponding $p$th center of mass. We obtained the results in the
table by a direct computer calculation using~\emph{polymake}~\cite
{Gawrilow2000vl}, which decomposes the polytopes into simplices and
computes the center of mass as a weighted sum of simplex centers of mass.
\end{pf}

It would be very interesting to get a general formula for these
polytope centers of mass.
\subsection{\texorpdfstring{Closed polygons in (rooted) spherical confinement.}{Closed polygons in (rooted) spherical confinement}}

Following the terminology of  Diao et al. \cite{Diao2011ie}, we say that a polygon $p$
is in rooted spherical confinement of radius $R$ if every vertex of the
polygon is contained in a sphere of radius $R$ centered at the first
vertex of the polygon. As a subspace of the space of closed polygons of
fixed edgelengths, the space of confined closed polygons inherits a~toric symplectic structure. In fact, the moment polytope for this
structure is a very simple subpolytope of the fan triangulation polytope.

\begin{definition}\label{defconfinedfantriangulationpolytope}
The \emph{confined fan polytope} $P_{n,R}(\vec{r}) \subset P_n(\vec
{r})$ is determined by the fan triangulation inequalities~(\ref{eqfanpolytope}) and the additional linear inequalities $d_i \leq R$.
\end{definition}

As before, we immediately have action-angle coordinates $P_{n,R}(\vec
{r}) \times T^{n-3}$ on the space of rooted confined polygons. We note
that the vertices of the confined fan triangulation polytope
corresponding to a space of confined polygons are \emph{not} all fixed
points of the torus action since this is not the entire moment
polytope; new vertices have been added by imposing the additional
linear inequalities. As before, we get criteria for sampling confined
polygons (directly analogous to Proposition~\ref{proppolygonsampling} for unconfined polygons).

\begin{proposition}
A polygon in $\widehat{\operatorname{Pol}}_3(n;\vec{r})$ is sampled
according to
the standard measure on  polygons in rooted spherical confinement of radius $R$
if and only if its diagonal lengths $d_1 = |v_1 - v_3|$, $d_2 = |v_1 -
v_4|, \ldots, d_{n-3} = |v_1 - v_{n-1}|$ are uniformly sampled from
the confined fan polytope $P_{n,R}(\vec{r})$ and its dihedral angles
around these diagonals are sampled independently and uniformly in
$[0,2\pi)$.
\end{proposition}

We can also compute expected values for chordlengths for confined
polygons following the lead of Theorem~\ref{thmexpectedchordlength}, but here our results are weaker because the p.d.f. of chordlength
is no longer simply a function of the number of edges skipped.

\begin{theorem}
The expected length of the chord joining vertex $v_1$ to vertex
$v_{k+1}$ in a polygon sampled according to the standard measure on
polygons in rooted spherical confinement of radius $R$ is given by
coordinate $d_{k-1}$ of the center of mass of the confined fan
triangulation polytope $P_{n,R}(\vec{r})$. For $n$ between $4$ and
$10$, $\vec{r} = \vec{1}$, and $R = 3/2$, these expectations are

{\fontsize{8.9pt}{10.9pt}\selectfont{
\[
\begin{array} {@{}l|c@{\hspace*{4pt}}c@{\hspace*{4pt}}c@{\hspace*{4pt}}c@{\hspace*{4pt}}c@{\hspace*{4pt}}c@{\hspace*{4pt}}c@{\hspace*{4pt}}c@{}}
n\setminus\raisebox{0.25em} {k} & 2 & 3 & 4 & 5 & 6 & 7 & 8 &
      (\textrm{denominator})
\\
\hline
\phantom{0}4 & 3/4 & & & & & & &
\\
\phantom{0}5 & 8/9 & 8/9 & & & & & &
\\
\phantom{0}6 & 293/336 & 316/336 & 293/336 & & & & &
\\
\phantom{0}7 & 281/320  & 298/320 & 298/320 & 281/320 & & & &
\\
\phantom{0}8 & 23{,}237 & 24{,}752 & 24{,}402 & 24{,}752 & 23{,}237 & & & 26{,}496
\\
\phantom{0}9 & 46{,}723 & 49{,}718 & 49{,}225 & 49{,}225 & 49{,}718 & 46{,}723 & & 53{,}256
\\
10 & 1{,}145{,}123 & 1{,}218{,}844 & 1{,}205{,}645 & 1{,}210{,}696 & 1{,}205{,}645 &
1{,}218{,}844 & 1{,}145{,}123 & 1{,}305{,}344
\end{array} %
\]}}%
where for $n = 8$, $9$, and $10$ we moved the common denominator of all
fractions in the row to the right-hand column.\vadjust{\goodbreak}
\end{theorem}

The proof is just the same as the proof of Theorem~\ref{thmexpectedchordlength}, and again we use \emph{polymake} for the computations.
The data show an interesting pattern: for 8, 9 and 10 edge polygons,
the confinement is tight enough that the data reveals small parity
effects in the expectations. For 10-gons, for instance, vertex $v_5$ is
on average closer to vertex $v_1$ than vertex $v_4$ is. We also
calculated the exact expectation of chordlength for equilateral 10-gons
confined to spheres of other radii. The results are shown in
Figure~\ref{figconfinedchordlengths}.

\begin{figure}

\includegraphics{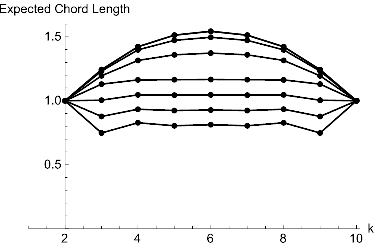}

\caption{Each line in this graph shows the expected length of the
chord joining vertices $v_1$ and $v_k$ in a random equilateral 10-gon.
The 10-gons are sampled from the standard measure on polygons in rooted
spherical confinement. From bottom to top, the confinement radii are
$1.25$, $1.5$, $1.75$, $2$, $2.5$, $3$, $4$ and $5$. Polygons confined
in a sphere of radius $5$ are unconfined. Note the small parity effects
which emerge in tighter confinement. These are exact expectations, not
the result of sampling experiments.}
\label{figconfinedchordlengths}
\end{figure}

\section{\texorpdfstring{Markov chain Monte Carlo for closed and confined random walks.}{Markov chain Monte Carlo for closed and confined random walks}}
\label{secMCMCmethods}

We have now constructed the action-angle coordinates on several spaces
of random walks, including closed walks, closed walks in rooted
spherical confinement, standard (open) random walks and open random
walks confined to half-spaces or slabs. In each case, the action-angle
coordinates have allowed us to prove some theorems and make some
interesting exact computations of probabilities on the spaces. To
address more complicated (and physically interesting) questions, we
will now turn to numerically sampling these spaces.

Numerical sampling of closed polygons underlies a substantial body of
work on the geometry and topology of polymers and biopolymers (see the
surveys of \cite{Orlandini2007kn} and \cite{Benham2005cl}, which
contain more than 200 references), which is a topic of interest in
statistical physics. Many of the physics questions at issue in these
investigations seem to be best addressed by computation. For instance,
while our methods above gave us simple (though not very tight)
theoretical bounds on the fraction of unknots among equilateral 6-gons,
a useful theoretical bound on, say, the fraction of unknots among
1273-gons seems entirely out of reach. On the other hand, it is
entirely reasonable to work on developing well-founded algorithms with
verified convergence and statistically defensible error bars for
experimental work on such questions, and that is precisely our aim in
this part of the paper.

\subsection{\texorpdfstring{Current sampling algorithms for random polygons.}{Current sampling algorithms for random polygons}}

A wide variety of sampling algorithms for random polygons have been
proposed. They fall into two main categories: Markov chain algorithms
such as polygonal folds~\cite{MR95g57016} or crankshaft moves~\cite
{Vologodskii1979ik,Klenin1988dt} (cf. \cite{Anonymous2010p2603} for
a discussion of these methods) and direct sampling methods such as the
``triangle method'' \cite{Moore2004ds} or the ``generalized
hedgehog'' method \cite{Varela2009cda} and the methods of Moore and
Grosberg \cite{Moore2005fh} and Diao, Ernst, Montemayor and Ziegler
\cite{Diao2011ie,Diaowt,Diao2012dza} which are both based on the
``sinc integral formula'' (\ref{eqrayleighform}).

Each of these approaches has some defects. No existing Markov chain
method has been proved to converge to the standard measure, though it
is generally conjectured that they do. It is unclear what measure the
generalized hedgehog method samples, while the triangle method clearly
samples a submanifold\footnote{It is hard to know whether this
restriction is important in practice. The submanifold may be
sufficiently ``well-distributed'' that most integrands of interest
converge anyway. Or perhaps calculations performed with the triangle
method are dramatically wrong for some integrands!} of polygon space.
The Moore--Grosberg algorithm is known to sample the correct
distribution, but faces certain practical problems. It is based on
computing successive piecewise-polynomial distributions for diagonal
lengths of a closed polygon and directly sampling from these
one-dimensional distributions. There is no problem with the convergence
of this method, but the difficulty is that the polynomials are high
degree with large coefficients and many almost-cancellations, leading
to significant numerical problems with accurately evaluating
them.\footnote{Hughes discusses these methods in Section~2.5.4 of his
book on random walks~\cite{hughes1995random}, attributing the formula
rederived by Moore and Grosberg~\cite{Moore2005fh} to a 1946 paper of
Treloar~\cite{TF9464200077}. The problems with evaluating these
polynomials accurately were known by the 1970s, when Barakat~\cite
{0301-0015-6-6-008} derived an alternate expression for this
probability density based on Fourier transform methods.} These problems
are somewhat mitigated by the use of rational and multiple-precision
arithmetic in~\cite{Moore2005fh}, but the number of edges in polygons
sampled with these methods is inherently limited. For instance, the
text file giving the coefficients of the polynomials needed to sample a
random closed 95-gon is over 25 megabytes in length. Diao et al.
avoid this problem by approximating these distributions by normals, but
this approximation means that they are not quite\footnote{Again, it is
unclear what difference this makes in practice.} sampling the standard
measure on polygon space.

\subsection{\texorpdfstring{The toric symplectic Markov Chain Monte Carlo algorithm.}{The toric symplectic Markov Chain Monte Carlo algorithm}}

We introduce a Markov Chain Monte Carlo algorithm for sampling toric
symplectic manifolds with an adjustable parameter $\beta\in(0,1)$
explained below. We will call this the
$\textsc{Toric-Symplectic-MCMC}(\beta)$ algorithm or
$\operatorname{TSMCMC}(\beta)$ for convenience. Though we intend to
apply this
algorithm to our random walk spaces, it works on any toric symplectic
manifold, so we state the results in this section and the next for an
arbitrary $2n$-dimensional toric symplectic manifold $M$ with moment
map $\mu\dvtx M \to\mathbb{R}^n$, moment polytope $P$, and action-angle
parametrization $\alpha\dvtx P \times T^n \to M$. The method is based on
a classical Markov chain for sampling convex regions of $\mathbb{R}^n$ called
the ``hit-and-run'' algorithm: choose a direction at random and sample
along the intersection of that ray with the region to find the next
point in the chain. This method was introduced by Boneh and Golan~\cite
{Boneh1979uy} and independently by Smith~\cite{Smith1980tu} as a
means of generating random points in a high-dimensional polytope. There
is a well-developed theory around this method which we will be able to
make use of below.

Since the action and angle variables are independent, we could resample
the angles every time we take a step in the Markov chain sampling
actions and the chain would certainly converge. However, it might not
be advantageous to do this: it does take some time to update the
angles, and if we are numerically integrating a functional which is
almost constant in the angles (a limiting case would be computing a
function of the chordlengths alone), this update would waste time. For
this reason, our algorithm has a parameter controlling the relative
rate of updates for the action and angle variables, called $\beta$.
At each step of TSMCMC$(\beta)$, with probability $\beta$ we update
the action variables by sampling the moment polytope $P$ using
hit-and-run and with probability $1-\beta$ we update the angle
variables by sampling the torus $T^n$ uniformly. When $\beta=
\frac{1}{2}$ this is analogous to the random scan
Metropolis-within-Gibbs samplers discussed by Roberts and
Rosenthal~\cite{Roberts1997ir} (see also~\cite{Latuszynski2013dz}).

\begin{codebox}\label{tsmcmc}
\Procname{$\textsc{Toric-Symplectic-MCMC}(\vec{p},\vec{\theta},\beta)$}
\zi$\mathit{prob} = \textsc{Uniform-Random-Variate}(0,1)$
\zi\If$\mathit{prob} < \beta$
\zi\Then\> \Comment Generate a new point in $P$ using the
hit-and-run algorithm.
\zi$\vec{v} = \textsc{Random-Direction}(n)$
\zi$(t_0,t_1) = \textsc{Find-Intersection-Endpoints}(P,\vec{p},\vec{v})$
\zi$t = \textsc{Uniform-Random-Variate}(t_0,t_1)$
\zi$\vec{p} = \vec{p} + t \vec{v}$

\zi\Else\> \Comment Generate a new point in $T^n$ uniformly.
\zi\For$\id{ind} = 1$ \To$n$
\zi\Do
$\theta_{\id{ind}} =\textsc{Uniform-Random-Variate}(0,2\pi)$
\End
\End
\zi\Return$(\vec{p},\vec{\theta})$
\end{codebox}

We now prove that the distribution of samples produced by this Markov
chain converges geometrically to the distribution generated by the
symplectic volume on~$M$. First, we show that the symplectic measure on
$M$ is invariant for TSMCMC.

To do so, recall that for any Markov chain $\Phi$ on a state space
$X$, we can define the $m$-step transition probability $\mathcal
{P}^m(x,A)$ to be the probability that an $m$-step run of the chain
starting at $x$ lands in the set $A$. This defines a measure $\mathcal
{P}^m(x,\cdot)$ on~$X$. The transition kernel $\mathcal{P} = \mathcal
{P}^1$ is called \emph{reversible} with respect to a probability
distribution $\pi$ if
%
\begin{equation}
\label{eqreversibility} \int_A \pi(   \mathrm{d}x)
\mathcal{P}(x,B) = \int_B \pi(   \mathrm{d}x) \mathcal
{P}(x,A)\qquad \mbox{for all measurable } A,B \subset X.\hspace*{-15pt}
\end{equation}
In other words, the probability of moving from $A$ to $B$ is the same
as the probability of moving from $B$ to $A$. If $\mathcal{P}$ is
reversible with respect to $\pi$, then $\pi$ is invariant for~$\mathcal{P}$: letting $A = X$ in (\ref{eqreversibility}), we see
that $\pi\mathcal{P} = \pi$.

In TSMCMC$(\beta)$, the transition kernel $\mathcal{P} = \beta
\mathcal{P}_1 + (1-\beta)\mathcal{P}_2$, where $\mathcal{P}_1$ is
the hit-and-run kernel on the moment polytope and $\mathcal{P}_2(\vec
{\theta},\cdot) = \tau$, where $\tau$ is the uniform measure on
$T^n$. Since hit-and-run is reversible on the moment polytope~\cite
{Smith1984vz} and since $\mathcal{P}_2$ is obviously reversible with
respect to $\tau$, we have the following.

\begin{proposition}\label{propreversibility}
\textup{TSMCMC}$(\beta)$ is reversible with respect to the symplectic measure
$\nu$ induced by symplectic volume on $M$. In particular, $\nu$ is
invariant for \textup{TSMCMC}$(\beta)$.
\end{proposition}

Recall that the total variation distance between two measures $\eta_1,
\eta_2$ on a state space~$X$ is given by
\[
\llvert \eta_1 - \eta_2 \rrvert _{\mathrm{TV}} := \sup
_{A \,
\mathrm{any}\,  \mathrm{measurable} \, \mathrm{set}} \bigl|\eta_1(A) - \eta_2(A)\bigr|.
\]
We can now prove geometric convergence of the sample measure generated
by TSMCMC($\beta$) to the symplectic measure in total variation distance.

\begin{theorem}\label{proptsmcmc-convergence}
Suppose that $M$ is a toric symplectic manifold with moment polytope
$P$ and action-angle coordinates $\alpha\dvtx P \times T^n \rightarrow
M$. Further, let $\mathcal{P}^m(\vec{p},\vec{\theta},\cdot)$ be
the $m$-step transition probability of the Markov chain given by
$\textsc{Toric-Symplectic-MCMC}(\beta)$ and let $\nu$ be the symplectic
measure on $M$.

There are constants $R < \infty$ and $\rho< 1$ so that for any $(\vec
{p},\vec{\theta}) \in\operatorname{int}(P) \times T^n$,
\[
\bigl\llvert \alpha_\star\mathcal{P}^m(\vec{p},\vec{
\theta},\cdot) - \nu \bigr\rrvert _{\mathrm{TV}} < R \rho^m.
\]
That is, for any choice of starting point, the pushforward by $\alpha$
of the probability measure generated by $\textsc{Toric-Symplectic-MCMC}(\beta)$ on $P \times T^{n}$ converges
geometrically (in the number of steps taken in the chain) to the
symplectic measure on $M$.
\end{theorem}

\begin{pf}
Let $\lambda$ be Lebesgue measure on the moment polytope $P$ and, as
above, let $\tau$ be uniform measure on the torus $T^n$. By
Theorem~\ref{theoremmp}, it suffices to show that
\[
\bigl\llvert \mathcal{P}^m(\vec{p},\vec{\theta}, \cdot) - \lambda
\times\tau \bigr\rrvert _{\mathrm{TV}} < R \rho^m.
\]
Since the transition kernels $\mathcal{P}_1$ and $\mathcal{P}_2$
commute, for any nonnegative integers $a$ and $b$ and partitions $i_1,
\ldots, i_k$ of $a$ and $j_1, \ldots, j_\ell$ of $b$ we have
%
\begin{eqnarray}
 \bigl(\mathcal{P}_1^{i_1}
\mathcal{P}_2^{j_1} \cdots\mathcal {P}_1^{i_k}
\mathcal{P}_2^{j_\ell} \bigr) (\vec{p},\vec{\theta },\cdot) &=&
\bigl(\mathcal{P}_1^a \mathcal{P}_2^b
\bigr) (\vec {p},\vec{\theta},\cdot)
\nonumber
\\[-8pt]
\label{eqkernelproduct}
\\[-8pt]
\nonumber
&=&  \mathcal{P}_1^a(
\vec{p},\cdot) \times \mathcal{P}_2^b(\vec{\theta},\cdot)
= \mathcal{P}_1^a(\vec {p},\cdot) \times\tau,
\end{eqnarray}
where the last equality follows from the fact that $\mathcal{P}_2(\vec
{\theta},\cdot) = \tau$ for any $\vec{\theta} \in T^n$.

The total variation distance between product measures is bounded above
by the sum of the total variation distances of the factors (this goes
back at least to Blum and Pathak~\cite{Blum1972bt}; see Sendler~\cite
{Sendler1975vn} for a proof), so we have that
%
\begin{eqnarray}
\nonumber
 \bigl\llvert \mathcal{P}_1^a(\vec{p},\cdot)
\times\mathcal {P}_2^b(\vec{\theta},\cdot) - \lambda
\times\tau \bigr\rrvert _{\mathrm{TV}}  &=& \bigl\llvert \mathcal{P}_1^a(
\vec{p},\cdot) \times\tau- \lambda\times\tau \bigr\rrvert _{\mathrm{TV}}
\\
\label{eqkerneltriangleineq} &\leq & \bigl\llvert \mathcal{P}_1^a(\vec
{p}, \cdot) - \lambda \bigr\rrvert _{\mathrm{TV}} + \llvert \tau- \tau \rrvert
_{\mathrm{TV}}
\\
\nonumber
& = & \bigl\llvert \mathcal{P}_1^a(\vec{p},
\cdot) - \lambda \bigr\rrvert _{\mathrm{TV}}.
\end{eqnarray}
Using \cite{Smith1984vz}, Theorem~3, the right-hand side is bounded
above by $ (1-\frac{\xi}{n2^{n-1}} )^{a-1}$ where $\xi$
is the ratio of the volume of $P$ and the volume of the smallest round
ball containing $P$. Let
\[
\kappa:= \biggl(1-\frac{\xi}{n2^{n-1}} \biggr).
\]
Then combining~(\ref{eqkernelproduct}),~(\ref{eqkerneltriangleineq}) and the binomial theorem yields
\begin{eqnarray*}
\bigl\llvert \mathcal{P}^m(\vec{p},\vec{\theta},\cdot) - \lambda
\times\tau \bigr\rrvert _{\mathrm{TV}} & =& \bigl\llvert \bigl(\beta
\mathcal{P}_1 + (1-\beta )\mathcal{P}_2
\bigr)^m(\vec{p},\vec {\theta},\cdot) - \lambda \times\tau \bigr
\rrvert _{\mathrm{TV}}
\\
& = &\Biggl\llvert \sum_{i=0}^m \pmatrix{m
\cr
i} \beta^{m-i}(1-\beta)^i \bigl(\mathcal{P}_1^{m-i}(
\vec{p},\cdot) \times\tau- \lambda \times\tau \bigr) \Biggr\rrvert _{\mathrm{TV}}
\\
& \leq &\sum_{i=0}^m \pmatrix{m
\cr
i}
\beta^{m-i}(1-\beta)^i \kappa ^{m-i-1}
\\
& = &\frac{1}{\kappa}\bigl(1+\beta(\kappa-1)\bigr)^m =
\frac{1}{\kappa} \biggl(1 - \frac{\beta\xi}{n2^{n-1}} \biggr)^m.
\end{eqnarray*}
The ratio $\xi$ of the volume of $P$ and the volume of smallest round
ball containing $P$ is always a positive number with absolute value
less than $1$, and hence $0 < 1 - \beta\xi/n2^{n-1} < 1$. This completes the proof.
\end{pf}

This proposition provides a comforting theoretical guarantee that
$\textsc{Toric}\mbox{-}\break \textsc{Symplectic-MCMC}(\beta)$ will eventually work. The proof
provides a way to estimate the constants $R$ and $\rho$. However, in
practice, these upper bounds are far too large to be useful. Further,
the rate of convergence for any given run will depend on the shape and
dimension of the moment polytope $P$ and on the starting position $x$.
There is quite a bit known about the performance of hit-and-run in
general theoretical terms; we recommend the excellent survey article of
Andersen and Diaconis~\cite{Andersen2007vn}. To give one example, L\`
ovasz and Vempala have shown~\cite{Lovasz2006gs} (see also~\cite
{Lovasz1999cq}) that the number of steps of hit-and-run required to
reduce the total variation distance between $\mathcal{P}^m(x,\cdot)$
and Lebesgue measure by an order of magnitude is proportional\footnote
{The constant of proportionality is large and depends on the geometry
of the polytope, and the amount of time required to reduce the total
variation distance to a fixed amount from the start depends on the
distance from the starting point to the boundary of the polytope.} to
$n^3$ where $n$ is the dimension of the polytope.

\subsection{\texorpdfstring{The Markov Chain CLT and Geyer's IPS error bounds for
TSMCMC integration.}{The Markov Chain CLT and Geyer's IPS error bounds for
TSMCMC integration}}
\label{subsecgeyerips}

We now know that the TSMCMC($\beta)$ algorithm will eventually sample
from the correct probability measure on any toric symplectic manifold,
and in particular from the correct probability distributions on closed
and confined random walks. We should pause to appreciate the
significance of this result for a moment---while many Markov chain
samplers have been proposed for closed polygons, none have been proved
to converge to the correct measure. Further, there has never been a
Markov chain sampler for closed polygons in rooted spherical
confinement (or, as far as we know, for slab-confined or half-space
confined arms).

However, the situation remains in some ways unsatisfactory. If we wish
to compute the probability of an event in one of these probability
spaces of polygons, we must do an integral over the space by collecting
sample values from a Markov chain. But since we do not have any
explicit bounds on the rate of convergence of our Markov chains, we do
not know how long to run the sampler, or how far the resulting sample
mean might be from the integral over the space. To answer these
questions, we need two standard tools: the Markov Chain Central Limit
Theorem and Geyer's Initial Positive Sequence (IPS) error estimators
for MCMC integration~\cite{Geyer1992wm}. For the convenience of
readers unfamiliar with these methods, we summarize the construction
here. Since this is basically standard material, many readers may wish
to skip ahead to the next section.

Combining Proposition~\ref{proptsmcmc-convergence} with~\cite
{Tierney1994uk}, Theorem~5 (which is based on \cite
{Cogburn1972uv}, Corollary~4.2) yields a central limit theorem for the
$\textsc{Toric-Symplectic}\mbox{-}\break \textsc{MCMC}(\beta)$ algorithm. To set notation,
suppose that a run of the TSMCMC$(\beta)$ algorithm produces the
sequence of points $((\vec{p}_0,\vec{\theta}_0),(\vec{p}_1,\vec
{\theta}_1), \ldots)$, where the initial point $(\vec{p}_0,\vec
{\theta}_0)$ is drawn from some initial distribution (e.g., a
delta distribution). For any run $R$, let
\[
\operatorname{SMean}(f;R,m) := \frac{1}{m} \sum
_{k=1}^m f(\vec {p}_k,\vec {
\theta}_k)
\]
be the sample mean of the values of a function $f\dvtx M \to\mathbb
{R}$ over
the first $m$ steps in~$R$.
We will use ``$f$'' interchangeably to refer to the original function
$f\dvtx M \to\mathbb{R}$ or its expression in action-angle
coordinates $f \circ
\alpha\dvtx P \times T^n \to\mathbb{R}$.

Let $E(f;\nu)$ be the expected value of $f$ with respect to the
symplectic measure $\nu$ on $M$. For each $m$ the normalized
sample error $\sqrt{m}(\operatorname{SMean}(f;R,m) - E(f;\nu))$
is a random variable (as it depends on the various random choices in
the run $R$).

\begin{proposition}\label{proptsmcmc-CLT}
Suppose $f$ is a square-integrable real-valued function on the toric
symplectic manifold $M$. Then regardless of the initial distribution,
there exists a real number $\sigma(f)$ so that
%
\begin{equation}
\label{eqtsmcmc-CLT} \sqrt{m}\bigl(\operatorname{SMean}(f;R,m) - E(f;\nu)\bigr)
\stackrel {w} {\longrightarrow} \mathcal{N}\bigl(0,\sigma(f)^2\bigr),
\end{equation}
where $\mathcal{N}(0,\sigma(f)^2)$ is the Gaussian distribution with
mean 0 and standard deviation $\sigma(f)$ and the superscript $w$
denotes weak convergence.
\end{proposition}

Given $\sigma(f)$ and a run $R$, the range $\operatorname
{SMean}(f;R,m) \pm
1.96 \sigma(f)/\sqrt{m}$ is an approximate $95\%$ confidence interval
for the true expected value $E(f;\nu)$. Abstractly, we can find
$\sigma(f)$ as follows.

The variance of the left-hand side of (\ref{eqtsmcmc-CLT}) is
\begin{eqnarray*}
&& m \operatorname{Var}\bigl(\operatorname{SMean}(f;R,m)\bigr) \\
&&\qquad=
\frac{1}{m} \sum_{i=1}^m
\operatorname{Var}\bigl(f(\vec {p}_i,\vec{\theta}_i)
\bigr) + \frac{1}{m} \sum_{i=1}^m
\mathop{\sum_{j=1}}_{j \neq i}^m
\operatorname{Cov} \bigl(f(\vec{p}_i,\vec{\theta}_i),f(
\vec{p}_j,\vec{\theta}_j)\bigr).
\end{eqnarray*}
Since the convergence in Proposition~\ref{proptsmcmc-CLT} is
independent of the initial distribution, $\sigma(f)$ will be the limit
of this quantity for \emph{any} initial distribution. Following Chan
and Geyer~\cite{Chan1994vh}, suppose the initial distribution is the
stationary distribution. In that case, the quantities
\[
\gamma_0(f) := \operatorname{Var}\bigl(f(\vec{p}_i,
\vec{\theta}_i)\bigr)
\]
and
\[
\gamma_k(f) := \operatorname{Cov}\bigl(f(\vec{p}_i,
\vec{\theta }_i),f(\vec {p}_{i+k},\vec{\theta}_{i+k})
\bigr)
\]
(the stationary variance and lag $k$ autocovariance, resp.) are
independent of $i$. Then
\[
\sigma(f)^2 = \lim_{m \to\infty} \Biggl(
\gamma_0(f) + 2 \sum_{k=1}^{m-1}
\frac{m-k}{m}\gamma_k(f) \Biggr) = \gamma_0(f) + 2
\sum_{k=1}^\infty\gamma_k(f)
\]
provided the sum on the right-hand side converges.

In what follows, it will be convenient to write the above as
%
\begin{equation}
\label{eqsigma1} \sigma(f)^2 = \gamma_0(f) + 2
\gamma_1(f) + 2 \sum_{k=1}^\infty
\Gamma_k(f),
\end{equation}
where $\Gamma_k(f) := \gamma_{2k}(f) + \gamma_{2k+1}(f)$. We
emphasize that the quantities $\gamma_0(f),\break   \gamma_k(f),\Gamma
_k(f)$ are associated to the \emph{stationary} Markov chain.

In practice, of course, these quantities, and hence this expression for
$\sigma(f)$ are not computable. After all, if we could sample directly
from the symplectic measure on $M$ there would be no need for TSMCMC.
However, as pointed out by Geyer~\cite{Geyer1992wm}, $\sigma(f)$ can
be estimated from the sample data that produced $\operatorname{SMean}(f;R,m)$.
Specifically, we will estimate the stationary lagged autocovariance
$\gamma_k(f)$ by the following quantity:
%
\begin{eqnarray}
\bar{\gamma}_k(f)& =&  \frac{1}{m} \sum
_{i=1}^{m-k}\bigl[f(\vec{p}_i,\vec {
\theta}_i) - \operatorname{SMean}(f;R,m)\bigr]
\nonumber
\\[-8pt]
\\[-8pt]
\nonumber
&& \hspace*{6pt}\qquad{}\times\bigl[f(
\vec{p}_{i+k},\vec {\theta }_{i+k})-\operatorname{SMean}(f;R,m)
\bigr]. 
\end{eqnarray}
Multiplication by $\frac{1}{m}$ rather than $\frac{1}{m-k}$
is not a typographical error (cf.~\cite{Geyer1992wm}, Section~3.1).
Let $\bar{\Gamma}_k(f) = \bar{\gamma}_{2k}(f) + \bar{\gamma
}_{2k+1}(f)$. Then for any $N > 0$
%
\begin{equation}
\label{eqsigma2} \bar{\sigma}_{m,N}(f)^2 := \bar{
\gamma}_0(f) + 2\bar{\gamma}_1(f) + 2\sum
_{k=1}^N \bar{\Gamma}_k(f)
\end{equation}
is an estimator for $\sigma(f)^2$. We expect the $\bar{\Gamma}_k$ to
decrease to zero as $k \to\infty$ since very distant points in the
run of the Markov chain should become statistically uncorrelated.
Indeed, since TSMCMC is reversible, Geyer shows this is true for the
stationary chain.

\begin{theorem}[(Geyer \cite{Geyer1992wm}, Theorem~3.1)]
\label{propGammak}
$\Gamma_k$ is strictly positive, strictly decreasing and strictly
convex as a function of $k$.
\end{theorem}

We expect, then, that any nonpositivity, nonmonotonicity, or
nonconvexity of the $\bar{\Gamma}_k$ should be due to $k$ being
sufficiently large that $\bar{\Gamma}_k$ is dominated by noise. In
particular, this suggests that a reasonable choice for $N$ in (\ref
{eqsigma2}) is the first $N$ such that $\bar{\Gamma}_N \leq0$, since
the terms past this point will be dominated by noise, and hence tend to
cancel each other.

\begin{definition}\label{defipse}
Given a function $f$ and a length-$m$ run of the TSMCMC algorithm as
above, let $N$ be the largest integer so that $\bar{\Gamma}_1(f),
\ldots, \bar{\Gamma}_N(f)$ are all strictly positive. Then the \emph
{initial positive sequence estimator} for $\sigma(f)$ is
\[
\bar{\sigma}_{m}(f)^2 := \bar{\sigma}_{m,N}(f)^2
= \bar{\gamma }_0(f) + 2\bar{\gamma}_1(f) + 2\sum
_{k=1}^N \bar{\Gamma}_k(f).
\]
\end{definition}

Slightly more refined initial sequence estimators which take into
account the monotonicity and convexity from Proposition~\ref
{propGammak} are also possible; see~\cite{Geyer1992wm} for details.

The pleasant result of all this is that $\bar{\sigma}_{m}$ is a
statistically consistent overestimate of the actual variance.

\begin{theorem}[(Geyer~\cite{Geyer1992wm}, Theorem~3.2)]
For almost all sample paths of TSMCMC,
\[
\liminf_{m \to\infty} \bar{\sigma}_{m}(f)^2
\geq\sigma(f)^2.
\]
\end{theorem}

Therefore, we propose the following procedure for \emph{Toric
Symplectic Markov Chain Monte Carlo integration} which yields
statistically consistent error bars on the estimate of the true value
of the integral.

\begin{tsmcmci*}\label{tsmcmci}
Let $f$ be a square-integrable function on a $2n$-dimensional toric
symplectic manifold $M$ with moment map $\mu\dvtx M \to\mathbb{R}^n$:
\begin{longlist}[8.]
\item[1.] Find the fixed points of the Hamiltonian torus action. The moment
polytope $P$ is the convex hull of the images of these fixed points
under $\mu$.
\item[2.] Convert this vertex description of $P$ to a halfspace
description. In other words, realize $P$ as the subset of points in
$\mathbb{R}
^n$ satisfying a collection of linear inequalities.\footnote{For small
problems, this can be done algorithmically~\cite
{Avis1992kb,Chazelle1993kn,Gawrilow2000vl}. Generally, this will
require an analysis of the moment polytope, such as the one performed
above for the moment polytopes of polygon spaces.}
\item[3.] Pick the parameter $\beta\in(0,1)$. We recommend repeating the
entire procedure for several short runs with various $\beta$ values to
decide on the best $\beta$ for a given application. The final error
estimate is a good measure of how well the chain has converged after a
given amount of runtime.
\item[4.] Pick a point $(\vec{p}_0, \vec{\theta}_0) \in P \times T^n$.
This will be the starting point of the Markov chain. Ideally, $\vec
{p}_0$ should be as far as possible from the boundary of $P$.
\item[5.] Using\vspace*{1pt} $(\vec{p}_0,\vec{\theta}_0)$ as the initial input,
iterate the TSMCMC$(\beta)$ algorithm\vspace*{1pt} for $m$ steps ($m \gg1$). This
produces a finite sequence $((\vec{p}_1,\vec{\theta}_1),\ldots,
(\vec{p}_m,\vec{\theta}_m))$ of points in $P \times T^n$.
\item[6.] Let $\operatorname{SMean}(f;m) = \frac{1}{m} \sum_{i=1}^m
f(\vec{p}_i,
\vec{\theta}_i)$ be the average value of $f$ over the run of points
produced in the previous step.
\item[7.] Compute the initial positive sequence estimator $\bar{\sigma}_{m}(f)^2$.
\item[8.] $\operatorname{SMean}(f;m)\pm1.96\bar{\sigma}_m(f)/\sqrt{m}$
is an
approximate $95\%$ confidence interval for the true expected value of
the function $f$.
\end{longlist}
\end{tsmcmci*}

\subsection{\texorpdfstring{Tuning the TSMCMC algorithm for closed and confined polygons.}{Tuning the TSMCMC algorithm for closed and confined polygons}}
\label{secnumerics}

For polygon sampling, the TSMCMC($\beta$) algorithm has several
adjustable parameters. We must always choose a starting polygon. For
unconfined polygons, we may choose any triangulation of the $n$-gon and
get a corresponding moment polytope. Finally, we must make an
appropriate choice of $\beta$. In this section, we report experimental
results which address these questions. In our experiments, we always
integrated total curvature and used equilateral closed polygons. At
least for unconfined polygons, we know the exact value of the
expectation from Theorem~\ref{thmexpectedtotalcurvature}. To
measure convergence, we used the Geyer IPS error estimate as a measure
of quality (lower is better). Since different step types take very
different amounts of~time to run, we ran different variations of the
algorithm for a consistent amount of CPU time, even though this led to
very different step counts.

We discovered in our experiments that the rate of convergence of
hit-and-run depends strongly on the start point. Our original choice of
start point---the regular planar equilateral $n$-gon---turned out to
be a very poor performer. While it seems like a natural choice
mathematically, the regular $n$-gon is tucked away in a corner of the
moment polytope and it takes hit-and-run quite a while to escape this
trap. After a number of experiments, the most consistently desirable
start point was obtained as follows. First, fold the regular $n$-gon
randomly along the diagonals of the given triangulation. Then,
borrowing an idea from Section~\ref{subsecptsmcmc}, randomly reorder
the resulting edge set (we will see below that this still results in a
closed, equilateral polygon). We used this as a starting configuration
in all of our unconfined experiments.

We also discovered that hit-and-run can converge relatively slowly when
sampling high-dimensional polytopes, leading to very long-range
autocorrelations in the resulting Markov chain. Following a suggestion
of Soteros~\cite{chrispc}, after considerable experimentation we
settled on the convention that a single ``moment polytope'' step in our
implementation of TSMCMC$(\beta)$ would represent ten iterations of
hit-and-run on the moment polytope. This reduced autocorrelations
greatly and led to better convergence overall. We used this convention
for all our numerical experiments below.

The $\textsc{Toric-Symplectic-MCMC}(\beta)$ algorithm depends on a
choice of triangulation $T$ for the $n$-gon to determine the moment
polytope $P$. There is considerable freedom in this choice, since the
number of triangulations of an $n$-gon is the Catalan number ${C_{n-2}
= \frac{1}{n-1} {2n-4 \choose n-2}}$ (\cite{Stanley1999eh}, Exercise
6.19). Using Stirling's approximation, this can be
approximated for large $n$ by ${C_{n-2} \sim4^{n-2}/(n-2)^{3/2} \sqrt
{\pi}}$ (\cite{Olver2010vy}, 26.5.6). We have proved above that the
$\textsc{Toric-Symplectic-MCMC}(\beta)$ algorithm will converge for any
of these triangulations, but the rate of convergence is expected to
depend on the triangulation, which determines the geometry of the
moment polytope. This geometry directly affects the rate of convergence
of hit-and-run; ``long and skinny'' polytopes are harder to sample than
``round'' ones (see Lovasz~\cite{Lovasz1999cq}).

To get a sense of the effect of the triangulation on the performance of
TSMCMC$(\beta)$, we set $\beta=0.5$ and $n = 23$ and ran the
algorithm from 20 start points for 20,000 steps. We then took the
average IPS error bar for expected total curvature over these 20 runs
as a measure of convergence. We repeated this analysis for 300 random
triangulations and 300 repeats of three triangulations that we called
the ``fan,'' ``teeth'' and ``spiral'' triangulations. The results are
shown in Figure~\ref{figtriangulations}. The definition of the fan
and teeth triangulations will be obvious from that figure; the spiral
triangulation is generated by traversing the $n$-gon in order
repeatedly, joining every other vertex along the traversal until the
triangulation is complete. Our experiments showed that this spiral
triangulation was the best performing triangulation among our
candidates, so we standardized on that triangulation for further
numerical experiments.

\begin{figure}

\includegraphics{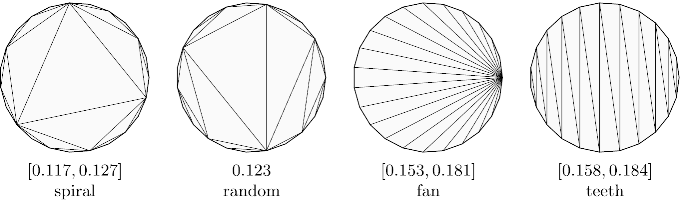}

%
\caption{We tested the average IPS $95\%$ confidence error estimate
for the expected value of total curvature over random equilateral
23-gons over 20 runs of the \textup{TSMCMC}(0.5) algorithm. Each run had a
starting point generated by folding and permuting a regular $n$-gon as
described above, and ran for 20,000 steps. We tried 300 random $n$-gons
and 300 repetitions of the same procedure for the ``spiral,'' ``fan,''
and ``teeth'' triangulations shown above. Below each triangulation is
shown the range of average error bars observed over 300 repetitions of
the 20-start-point trials; for the random triangulation we report the
best average error bar over a single 20-start-point-trial observed for
any of the 300 random triangulations we computed. We can see that the
algorithm based on the spiral triangulation generally outperforms
algorithms based on even the best of the 300 random triangulations,
while algorithms based on the fan and teeth triangulations converged
more slowly.}
\label{figtriangulations}
\end{figure}

We then considered the effect of varying the parameter $\beta$ for the\break 
TSMCMC($\beta$) algorithm using the spiral triangulation. We ran a
series of trials computing expected total curvature for 64-gons over 10
minute runs, while varying $\beta$ from $0.05$ (almost all dihedral
steps) to $0.95$ (almost all moment polytope steps) over 10 minute
runs. We repeated each run 50 times to get a sense of the variability
in the Geyer IPS error estimators for different runs. Since dihedral
steps are considerably faster than moment polytope steps, the step
counts varied from about 1 to 9 million. The resulting Geyer IPS error
estimators are shown in Figure~\ref{figspiralgridone}. Our
recommendation is to use the spiral triangulation and $\beta= 0.5$ for
experiments with unconfined polygons. From the 50 runs using the
recommended $\beta= 0.5$, the run with the median IPS error estimate
produced an expected total curvature estimate of $101.724 \pm0.142$
using about $4.6$ million samples; recall that we computed in
Table~\ref{tabexpectedtotalcurvature} that the expected value of
total curvature for equilateral, unconfined $64$-gons is a complicated
fraction close to $101.7278$.

\begin{figure}

\includegraphics{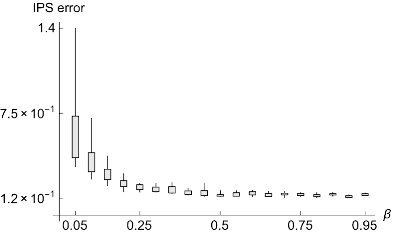}

\caption{The figure above shows a box-and-whisker plot for the IPS
error estimators observed in computing expected total curvature over 50
runs of the \textup{TSMCMC}($\beta$) algorithm for various values of $\beta$.
The boxes show the $\frac{1}{4}$ to $\frac{3}{4}$\vspace*{1pt} quantiles
of the data, while the whiskers extend from the $0.05$ quantile to the
$0.95$ quantile. While the whiskers show that there is plenty of
variability in the data, the general trend is that the performance of
the algorithm improves rapidly as $\beta$ varies from $0.05$ to
$0.25$, modestly as $\beta$ varies from $0.25$ to $0.5$ and is
basically constant for $\beta$ from $0.5$ to $0.95$.}
\label{figspiralgridone}
\end{figure}

\subsection{\texorpdfstring{Crankshafts, folds and permutation steps for unconfined
equilateral polygons.}{Crankshafts, folds and permutation steps for unconfined
equilateral polygons}}
\label{subsecptsmcmc}

It is an old observation that the space of closed equilateral $n$-gons
has an action of the permutation group $S_n$ given by permuting the
edges. For instance, the ``triangle method'' of Moore, Lua and
Grosberg~\cite{Moore2004ds} is based on this idea. Since all edges
are the same length, a reordered polygon is clearly still equilateral.
It is also closed: the end-to-end displacement of the polygon is the
vector sum of the edges, which is invariant under reordering. It seems
desirable,\vspace*{1pt} but not entirely obvious, that this action preserves the
probability measure on $\widehat{\operatorname{Pol}}_3(n;\vec{1})$.

\begin{lemma}
The action of the permutation group $S_n$ on $\widehat{\operatorname
{Pol}}_3(n;\vec{1})$ given by
reordering the edges preserves the standard measure.
\end{lemma}

\begin{pf}
$\!\!$By permuting coordinates, the symmetric group acts on the $n$-fold
product of spheres $\operatorname{Arm}_3(n;\vec{1})= S^2(1) \times
\cdots\times S^2(1)$ by
isometries. This  descends to an action by isometries on the
Riemannian submanifold \mbox{$\operatorname{Pol}_3(n;\vec{1})\subset
\operatorname{Arm}_3(n;\vec{1})\!$} since we have already
seen that $\operatorname{Pol}_3(n;\vec{1})$ is invariant under the
action of $S_n$. Though a
measure-preserving action on a space generally does not preserve
Hausdorff measure on subspaces of lower dimension, the condition that
this action is by isometries is quite strong, and does imply that the
restriction of this action to $\operatorname{Pol}_3(n;\vec{1})$ is
measure-preserving there. It
is then standard that the corresponding action on the quotient\vspace*{1pt} space
$\widehat{\operatorname{Pol}}_3(n;\vec{1})= \operatorname
{Pol}_3(n;\vec{1})/\operatorname{SO}(3)$ is measure-preserving there
because $\widehat{\operatorname{Pol}}_3(n;\vec{1})$
has the pushfoward measure.
\end{pf}

As a consequence, we will see that we can mix permutation steps with
standard TSMCMC steps without losing geometric convergence or the
applicability of the central limit theorem. Such a Markov chain is a
mixture of dihedral angle steps, moment polytope steps, and permutation
steps in some proportion. It is interesting to note that we can recover
algorithms very similar to the standard ``crankshaft'' and ``fold''
Markov chains by allowing no moment polytope steps in the chain.

Since previous authors have observed that adding permutation steps can
significantly speed up convergence in polygon samplers~\cite
{Anonymous2010p2603}, we now experiment to see whether our algorithm,
too, can be improved by mixing in some permutations. More precisely, we
can define a new Markov chain $\textsc{Polygon-Permutation}$ on $\widehat
{\operatorname{Pol}}_3(n;\vec{1})$
by permuting edges at each step:

\begin{codebox}
\Procname{\textsc{Polygon-Permutation}(pol)}
\zi$\sigma= \textsc{Uniform-Permutation}(n)$
\zi pol $ = \textsc{Permute-Edges}(\mathrm{pol},\sigma)$
\zi\Return pol
\end{codebox}

Since the symplectic measure on $\widehat{\operatorname
{Pol}}_3(n;\vec{1})$ is permutation-invariant, the
symplectic measure is stationary for $\textsc{Polygon-Permutation}$.

Now, we can mix TSMCMC$(\beta)$ with \textsc{Polygon-Permutation} to
get the following \textsc{Permutation-Toric-Symplectic-MCMC}$(\beta
,\delta)$ algorithm, where $\delta\in[0,1)$ gives the probability of
doing a permutation\vspace*{1pt} step rather than a TSMCMC$(\beta)$ step. Recall
that ${\alpha\dvtx P \times T^{n-3} \to\widehat{\operatorname
{Pol}}_3(n;\vec{1})}$ is the action-angle
parametrization, where $P$ is the moment polytope induced by the chosen
triangulation.

\begin{codebox}
\Procname{$\textsc{Permutation-Toric-Symplectic-MCMC}(\vec{p},\vec
{\theta},\beta,\delta)$}
\zi$\mathit{prob} = \textsc{Uniform-Random-Variate}(0,1)$
\zi\If$\mathit{prob} < \delta$
\zi\Then$(\vec{p},\vec{\theta} ) = \alpha^{-1}(\textsc
{Polygon-Permutation}(\alpha(\vec{p},\vec{\theta} )))$
\zi\Else$(\vec{p},\vec{\theta} ) = \textsc
{Toric-Symplectic-MCMC}(\vec{p},\vec{\theta},\beta)$
\End
\zi\Return$(\vec{p},\vec{\theta} )$
\end{codebox}

Although \textsc{Polygon-Permutation} is \emph{not} ergodic, the fact
that it is stationary with respect to the symplectic measure is, after
combining Proposition~\ref{proptsmcmc-convergence} and~\cite{Tierney1994uk}, Proposition~3, enough to imply that \textsc
{Permutation-Toric-Symplectic-MCMC}$(\beta, \delta)$ is (strongly)
uniformly ergodic.

\begin{proposition}\label{propptsmcmcergodicity}
Let $\widehat{\mathcal{P}}$ be the transition kernel for
\textup{PTSMCMC}$(\beta,\delta)$ with $0 < \beta< 1$ and $\delta< 1$ and let
$\nu$ be the symplectic measure on $\widehat{\operatorname
{Pol}}_3(n;\vec{1})$. Then there exist
constants $R < \infty$ and $\rho< 1$ so that for any $(\vec{p},
\vec{\theta} ) \in\operatorname{int}(P) \times\break  T^{n-3}$,
\[
\bigl\llvert \alpha_\star\widehat{\mathcal{P}}^m(\vec{p},
\vec{\theta },\cdot) - \nu \bigr\rrvert _{\mathrm{TV}} < R \rho^m.
\]
\end{proposition}

Just as in Proposition~\ref{proptsmcmc-CLT}, since PTSMCMC$(\beta
,\delta)$ is uniformly ergodic and reversible with respect to
symplectic measure, it satisfies a central limit theorem.

\begin{proposition}\label{propptmcmcCLT}
Suppose $f\dvtx\widehat{\operatorname{Pol}}_3(n;\vec{1})\to\mathbb
{R}$ is square-integrable. For any run $R$ of
\textup{PTSMCMC}$(\beta,\delta)$, let $\operatorname{SMean}(f;R,m)$ be the
sample mean
of the value of $f$ over the first $m$ steps of $R$. Then there exists
a real number $\sigma(f)$ so that
\[
\sqrt{m}\bigl(\operatorname{SMean}(f;R,m) - E(f;\nu)\bigr) \stackrel {w} {
\longrightarrow} \mathcal{N}\bigl(0,\sigma(f)^2\bigr).
\]
\end{proposition}

The rest of the machinery of Section~\ref{subsecgeyerips}, including
the initial positive sequence estimator for $\sigma(f)^2$, also
applies. As a consequence, we get a modified \textit{Toric Symplectic
Markov Chain Monte Carlo integration} procedure adapted to unconfined,\vspace*{1pt}
equilateral polygons.
Note that the full symmetric group $S_n$ does not act on $\widehat
{\operatorname{Pol}}_3(n;\vec{r})$ when
not all $r_i$ are equal, so PTSMCMC$(\beta, \delta)$ cannot be used
to sample nonequilateral polygons. Reordering the edges of a polygon in
$\widehat{\operatorname{Pol}}_3(n;\vec{r})$ by $\sigma\in S_n$
would\vspace*{1pt} still\vspace*{1pt} yield a closed polygon, but the
new polygon would belong to a different space: $\widehat{\operatorname{Pol}
}_3(n;\sigma\cdot\vec{r})$. However, when many edgelengths are
equal, a subgroup\vspace*{1pt} of the symmetric group which permutes only those
edges certainly acts on $\widehat{\operatorname{Pol}}_3(n;\vec{r})$.
We recommend making use of this
smaller set of permutations when possible. Permuting edges never
preserves spherical confinement, so PTSMCMC$(\beta,\delta)$ is
inapplicable to confined polygon sampling.

Having defined PTSMCMC$(\beta,\delta)$ and settled on a canonical
starting point (the folded, permuted regular $n$-gon) and triangulation
(the spiral), it remains to decide on the best values of $\beta$ and
$\delta$. The question is complicated by the fact that the three
different types of steps---permutations, folding steps and
moment-polytope hit-and-run steps---take different amounts of CPU
time. To attempt to evaluate the various possibilities fairly, we ran
experiments computing the expected total curvature for 64-gons where
each experiment ran for $10$ minutes of CPU time, completing between
$2$ million and $15$ million steps depending on the mixture of step
types. We measured the $95\%$ confidence IPS error bars for each run,
producing the data in Figure~\ref{figspiralgrids}, and used the size
of this error bar as a measure of convergence.

\begin{figure}[t]

\includegraphics{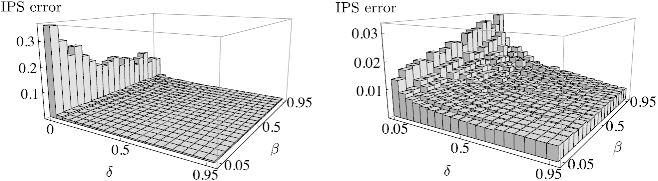}

\caption{This plot shows the IPS error estimator for the average total
curvature of unconfined equilateral 64-gons. The IPS error was computed
for 10-minute runs of the \textup{PTSMCMC}($\beta,\delta$) Markov chain
algorithm. The values of $\delta$ (the fraction of permutations among
all steps) ranged from $0$ to $0.95$ in steps of $0.05$ in the figure
on the left, and from $0.05$ to $0.95$ in the figure on the right. In
both plots, the values of $\beta$ (the fraction of moment polytope
steps among nonpermutation steps) ranged from $0.05$ to $0.95$ in steps
of $0.05$. When $\delta= 0$, this is just the \textup{TSMCMC}($\beta$) chain;
these are the comparatively very large error estimates in the back row
of the left figure. Removing those runs yields the plot on the right.
We observed that convergence was very sensitive to $\delta$, with
error bars improving dramatically as soon as the fraction of
permutation steps becomes positive: even the worst \textup{PTSMCMC}($\beta,\delta$) run with $\delta> 0$ had error bars 3 times smaller than
the error bars of the best \textup{TSMCMC}($\beta$) run. From the view at
right, we can see that the error bars continue to improve more modestly
as $\delta$ increases. Varying $\beta$ has little effect on the error
estimate when $\delta$ is large.}
\label{figspiralgrids}
\end{figure}

The data in Figure~\ref{figspiralgrids} show that the fraction
$\delta$ of permutation steps is the most important factor in
determining the rate of convergence in the PTSMCMC($\beta,\delta$)
algorithm. This shows that the extra complication in defining
PTSMCMC($\beta,\delta$) for unconfined equilateral polygons is
worth it: the error bars produced by PTSMCMC($\beta,\delta$) to
compute the expected total curvature of unconfined equilateral
$64$-gons are anywhere from 3 to 30 times smaller than the error bars
for TSMCMC($\beta$).

Larger values of $\beta$ produce smaller error bars when $\delta= 0$,
meaning that a large fraction of moment polytope steps are needed to
produce mixing when there are no permutation steps. On the other hand,
as we can see in Figure~\ref{figdelta05normalizederror}, even when
$\delta= 0.05$ the permutation steps provide enough mixing that $\beta
$ has virtually no effect on the IPS standard deviation estimator. In
this case, the effect of $\beta$ on the size of the error bars is due
to the fact that dihedral steps are faster than moment polytope steps,
so runs with small $\beta$ produce more samples, and hence smaller
error bars.

\begin{figure}[b]

\includegraphics{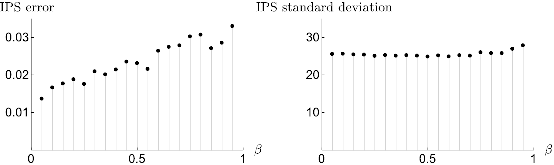}

\caption{These plots show the IPS error estimator and the IPS standard
deviation estimator for the average total curvature of unconfined
equilateral 64-gons using 10-minute runs of \textup{PTSMCMC}($\beta$, $0.05$).
Although the IPS error estimate decreases as $\beta$ decreases, the
plot on the right demonstrates that the IPS standard deviation
estimator is essentially constant---and presumably close to the true
standard deviation of the total curvature function---across the
different values of $\beta$. Since the IPS error estimate is
proportional to the standard deviation estimate divided by the square
root of the number of samples, we can see that the variation in IPS
error bars for these runs is almost entirely due to the difference in
the number of samples.}
\label{figdelta05normalizederror}
\end{figure}

Once $\delta$ is large, varying $\beta$ seems to have little effect
on the convergence rate. In fact, though our theory above no longer
proves convergence, we seem to get a very competitive algorithm by
removing moment polytope steps altogether ($\beta= 0$) and performing
only permutations and dihedral steps. This algorithm corresponds to the
``fold or crankshaft with permutations mixed in'' method.

In practice, we make a preliminary recommendation of $\delta=0.9$ and
$\beta=0.5$ for experimental work. These parameters guarantee
convergence (by our work above) while optimizing the convergence rate.
Using these recommended parameters, a 10-minute run of
PTSMCMC($0.5,0.9$) for unconfined, equilateral 64-gons produced just
under $7$ million samples and an expected total curvature of $101.7276
\pm0.0044$, which compares quite favorably to the actual expected
total curvature of $101.7278$.

We observed that the absolute error in our computations of expected
total curvature was less than our error estimate in $361$ of $380$ runs
($95\%$), which is exactly what we would expect from a $95\%$
confidence value estimator. We take this as solid evidence that the
Markov chain is converging and the error estimators are working as expected.


\subsection{\texorpdfstring{Calculations on confined polygons.}{Calculations on confined polygons}} 
\label{subcalculationsonconfinedpolygons}
Recall from~Definition~\ref{defconfinedfantriangulationpolytope}
that a polygon is in spherical confinement in a sphere of radius $R$
centered at vertex $v_1$ of the polygon if the vector $\vec{d}$ of fan
diagonals of the polygon lies in the confined fan polytope
$P_{n,R}(\vec{r})$. This means that we can sample such polygons
uniformly by restricting the hit-and-run steps in TSMCMC($\beta$) to
the confined fan polytope $P_{n,R}(\vec{r})$.

We again only explored the situation for equilateral polygons of edgelength one. After some experimentation, we settled on the ``folded
triangle'' as a start point. This polygon is constructed by setting
each diagonal length $d_i$ to one and choosing dihedrals randomly. This
polygon is contained in spherical confinement for every $R \geq1$, so
we could use it for all of our experiments. We investigated 23-gons
confined to spheres of radius $2$, $4$, $6$, $8$, $10$ and $12$,
measuring the Geyer IPS error estimate for values of $\beta$ selected
from $0.05$ (almost all dihedral steps) to $0.95$ (almost all moment
polytope steps) over 10-minute runs. Again, since dihedral steps are
faster to run than moment polytope steps, the step counts varied over
the course of the experiments. For instance, in the radius 2
experiments, we observed step counts as high as 35 million and as low
as 7 million over runs with various $\beta$ values. Our integrand was
again total curvature. Since we do not have an exact solution for the
expected total curvature of a confined $n$-gon, we were unable to check
whether the error bars predicted actual errors. However, it was
comforting to note that the answers we got from runs with various
parameters were very consistent. We ran each experiment 50 times to get
a sense of the repeatability of the Geyer IPS error bar; the results
are shown in Figure~\ref{figconfinedgrids}.

\begin{figure}[t]

\includegraphics{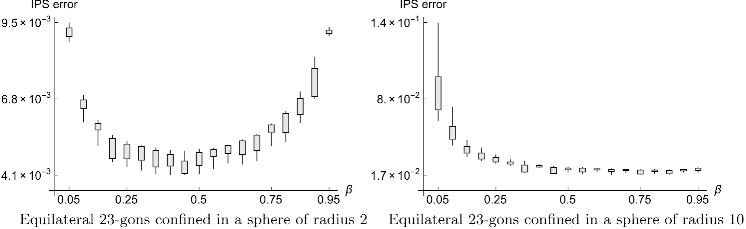}

\caption{These box-and-whisker plots show the results of computing the
expected total curvature for confined equilateral $23$-gons with
edgelength $1$. The confinement model here is ``rooted spherical
confinement,'' meaning that each vertex is within a specified distance
of the first vertex. For each $\beta$ value, we repeated $10$-minute
experiments $50$ times, computing $50$ values for the Geyer IPS
estimator. The boxes show the second and third quartiles of these $50$
data points, while the whiskers show the $0.05$ to $0.95$ quantiles of
the IPS estimators observed.}\vspace*{-6pt}
\label{figconfinedgrids}
\end{figure}

We observed first that there is a clear trend in the error bar data.
For the tightly confined runs, there was a noticeable preference for
$\beta\sim0.5$, while in less tight confinement the results generally
continued to improve modestly as $\beta$ increased. Still, we think
the data supports a general recommendation of $\beta= 0.5$ for future
confined experiments, with a possible decrease to $\beta= 0.4$ in very
tight confinement, and this is our recommendation to future investigators.

A very striking observation from Figure~\ref{figconfinedgrids} is
that the error bars for the tightly confined $23$-gons in a sphere of
radius $2$ are about 10 times smaller than the error bars for the very
loosely confined $23$-gons in a sphere of radius $10$. That is, our
algorithm works better when the polygon is in tighter confinement. In
some sense, this is to be expected, since the space being sampled is
smaller. However, it flies in the face of the natural intuition that
confined sampling should be numerically more difficult than unconfined sampling.

Using TSMCMC($0.5$), we computed the expected total curvature of
tightly confined equilateral 50- and 90-gons. Those expectations are
shown in Table~\ref{tabconfinedtotalcurvature}. We can compare
these data directly by looking at expected turning angles as in
Figure~\ref{figconfinedturningangle}. In this very tight
confinement regime, the effect of confinement radius on expected
turning angle dominates the effect of the number of edges.

\begin{table}[t]
\tabcolsep=18pt
\caption{This table shows the expected total curvature of equilateral
50- and 90-gons in rooted spherical confinement. We sampled equilateral
50- and 90-gons in confinement radii from $1.1$ to $1.6$ using
20-minute runs of \textup{TSMCMC}($0.5$) and computed the average total
curvature and IPS error bars for each run. Each 50-gon run yielded
about 14.5 million samples, while each 90-gon run yielded about 8
million samples. The bottom line shows the exact expectation of total
curvature for unconfined polygons given by Theorem~\protect\ref{thmexpectedtotalcurvature}. More extensive information on expectations of
confined total curvatures has been computed by Diao, Ernst, Montemayor
and Ziegler~\protect\cite{clauspc}}
\label{tabconfinedtotalcurvature}
\begin{tabular*}{\tablewidth}{@{\extracolsep{\fill}}lcc@{}}
\hline
& \multicolumn{2}{c@{}}{\textbf{Expected total curvature of
tightly-confined}}
\\
\multicolumn{1}{@{}l}{\multirow{2}{50pt}[-7pt]{\textbf{Confinement radius}}}& \multicolumn{2}{c@{}}{\textbf{equilateral 50- and
90-gons}}    \\[-4pt]
& \multicolumn{2}{l@{}}{\hrulefill}\\
 & \textbf{50-gons} & \textbf{90-gons} \\
\hline
$1.1$ & $103.1120 \pm0.0093$ & $185.701 \pm0.028$ \\
$1.2$ & $100.1900 \pm0.0089$ & $180.261 \pm0.028$ \\
$1.3$ & $\phantom{1}97.8369 \pm0.0088$ & $175.947 \pm0.028$ \\
$1.4$ & $\phantom{1}95.8891 \pm0.0090$ & $172.346 \pm0.027$ \\
$1.5$ & $\phantom{1}94.1979 \pm0.0091$ & $169.271 \pm0.028$ \\
$1.6$ & $\phantom{1}92.7501 \pm0.0094$ & $166.660 \pm0.029$ \\
$\infty$ & $\phantom{1}79.74197470$\phantom{0000} & $142.5630093$\phantom{000}\\
\hline
\end{tabular*}
\end{table}

\begin{figure}[b]

\includegraphics{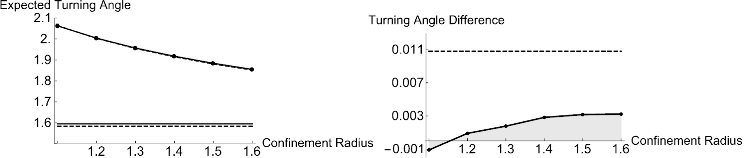}

\caption{The plot on the left shows the expected turning angles of
equilateral 50-gons (solid) and equilateral 90-gons (dashed) in rooted
spherical confinement of radii from $1.1$ to $1.6$. The horizontal
lines show the expected turning angles for unconfined 50- and 90-gons
computed using Theorem~\protect\ref{thmexpectedtotalcurvature},
which are
${\simeq}1.59484$ and ${\simeq}1.58403$, respectively. The plot on the
right shows the differences between the expected turning angles of
equilateral 50-gons and the expected turning angles of equilateral
90-gons. The black dots show this difference for various confinement
radii, while the dashed line shows the corresponding difference for
unconfined polygons. Without confinement, we expect polygons with more
edges to have smaller expected turning angle, since each individual
edge feels less pressure to get back to the starting point. These data
provide evidence this effect dissipates and even reverses in extremely
tight confinement.}
\label{figconfinedturningangle}
\end{figure}

\section{\texorpdfstring{Comparison with existing work, conclusion and future directions.}{Comparison with existing work, conclusion and future directions}}

Now that we have laid out the symplectic theory of random walks and a
few of its consequences, it is time to look back and see how we can
reconcile this point of view with the existing understanding of closed
random walks. In the methods of Moore and Grosberg~\cite{Moore2005fh}
and Diao et al.~\cite{Diao2011ie}, closed random walks are generated
incrementally,\vadjust{\goodbreak} using distributions derived from
the p.d.f. $\Phi_n(\vec{\ell} )$ given in~(\ref{eqftcpdfv2}) for
the end-to-end distance of a random walk of $n$ steps. To review, the
key idea is that if we have taken $m-1$ steps of an $n$-step closed
walk and arrived at the $m$th vertex~$\vec{v}_m$, the p.d.f. of the next
vertex $\vec{v}_{m+1}$ (conditioned on the steps we have already
taken) is given by
\[
P(\vec{v}_{m+1}|\vec{v}_1, \ldots, \vec{v}_m) =
\frac{\Phi
_{1}(\vec{v}_{m+1} - \vec{v}_m) \Phi_{n-m-1}(\vec{v}_{m+1} - \vec
{v}_1)}{\Phi_{n-m}(\vec{v}_m - \vec{v}_1)},
\]
which is some complicated product of piecewise-polynomial $\Phi_k(\vec
{\ell} )$ functions. We can sample $\vec{v}_{m+1}$ from this
distribution, and hence generate the rest of the walk iteratively.

From the moment polytope point of view, the situation is considerably
simpler. First, we observe that everything in the equation above can be
expressed in terms of diagonal lengths in the fan triangulation
polytope, since the length of the vector $\vec{\ell}$ is the only
thing that matters in the formula for $\Phi_k(\vec{\ell})$. If we
let $\vec{v}_1 = \vec{0}$ by convention, then conditioning on $\vec
{v}_1, \ldots, \vec{v}_m$ is simply restricting our attention to the
slice of the moment polytope given by setting the diagonal lengths $d_1
= |\vec{v}_3|, d_2 = |\vec{v}_4|, \ldots, d_{m-2} = |\vec{v}_m|$.
The p.d.f. $P(\vec{v}_{m+1}|\vec{v}_1, \ldots, \vec{v}_m)$ is then the
projection of the measure on this slice of the moment polytope to the
coordinate $d_{m-1}$. This distribution is piecewise-polynomial
precisely because it is the projection of Lebesgue measure on a convex
polytope with a finite number of faces.

Of course, projecting volume measures of successive slices to
successive coordinates is a perfectly legitimate way to sample a convex
polytope, which is another explanation for why these methods work; they
are basically sampling successive marginals of the coordinate
distributions on a succession of smaller convex polytopes. By contrast,
our method generates the entire vector of diagonal lengths $d_1, \ldots
, d_{n-3}$ simultaneously according to their joint distribution by
sampling the moment polytope directly. More importantly, it offers a
geometric insight into what this joint distribution \emph{is} which
seems like it would be very hard to develop by analyzing~(\ref{eqftcpdfv2}).

In conclusion, the moment polytope picture offers a clarifying and
useful perspective on closed and confined random walks. It is clear
that we have only scratched the surface of this topic in this paper,
and that many fascinating questions remain to be explored both
theoretically and computationally. In the interest of continuing the
conversation, we provide an unordered list of open questions suggested
by the work above.
\begin{itemize}
\item Previous studies of the relative efficiency of polygon sampling
algorithms have focused on minimizing pairwise correlations between
edges as a measure of performance. Proposition~\ref{proppolygonsampling} suggests a more subtle  approach to evaluating
sample quality: measure the uniformity of the distribution of diagonal
lengths over the moment polytope and of dihedral angles over the torus
(cf.~\cite{Mardia2000um}).
\item It remains open to try to extend these methods to prove that a
chain consisting only of permutation and dihedral steps is still
strongly geometrically convergent on unconfined equilateral polygon
space. This would lead directly to a proof of convergence for the
crankshaft and fold algorithms, and hence place many years of sampling
experiments using these methods on a solid theoretical foundation.
\item Can we use the moment polytope pictures above for confined
polygons to prove theorems about polygons in confinement? For instance,
it would be very interesting to show that the expectation of total
curvature is monotonic in the radius of confinement.
\item What is the corresponding picture for random planar polygons? Of
course, we can see the planar polygons as a special slice of the
action-angle coordinates where the angles are all zero or $\pi$. But
is it true that sampling this slice according to Hausdorff measure in
action-angle space corresponds to sampling planar polygons according to
their Hausdorff measure inside space polygons?\footnote{These questions
are less obvious than they may appear at first glance: the cylindrical
coordinates $\theta$ and $z$ are action-angle coordinates on the
sphere, but it is not the case that the arclength measure on a curve in
the $\theta$--$z$ cylinder pushes forward to the arclength measure on
the image of the curve on the sphere, even though the area measure on
the $\theta$--$z$ cylinder does push forward to the standard area
measure on the sphere.} If not, can we correct the measure somehow? Or
is there another picture for planar polygons entirely?
\item Can we understand the triangulation polytopes better? Can we
compute their centers of mass explicitly, for example? It is well known
that finding the center of mass of a high-dimensional polytope is
algorithmically difficult, so we cannot hope for a purely mechanical
solution to the problem. But a deeper understanding of these polytopes
seems likely to result in interesting probability theorems.
\item Why are permutation steps so effective in the PTSMCMC Markov
chain? It seems easy to compute that the number of points in the
permutation group orbit of an $n$-edge polygon is growing much faster
than the volume of equilateral polygon space computed by~\cite
{Takakura2001ir,Khoi2005ch,Mandini2008wq} and given above as
Corollary~\ref{corequilateralvolume}. Can we prove that the points
in this orbit are usually well distributed over polygon space? This
would give an appealing proof of the effectiveness of Grosberg's
triangle method for polygon sampling~\cite
{Moore2004ds,Moore2004tw,Lua2004wa}.
\item There is a large theory of ``low-discrepancy'' or
``quasi-random'' sequences on the torus which can provide better
results in numerical integration than uniform random sampling. Would it
be helpful to choose our dihedrals from such a sequence in the
integration method above?
\item Now that we can sample confined polygons quickly, with solid
error bars on our calculations, what frontiers does this open in the
numerical study of confined polymers? We take our cues from the
pioneering work of Diao, Ernst, Montemayor and Ziegler \cite
{Diao2011ie,Diaowt,Diao2012dza,diaopreprint}, but are eager to
explore this new experimental domain. For instance, sampling tightly
confined $n$-gons might be a useful form of ``enriched sampling'' in
the hunt for complicated knots of low equilateral stick number, since
very entangled polygons are likely to be geometrically compact as well.
\end{itemize}
We introduced a related probabilistic theory of nonfixed edgelength
closed polygons in a previous paper~\cite{CPACPA21480} by relating
closed polygons with given total length to Grassmann manifolds. It
remains to explain the connection between that picture and this one,
and we will take up that question shortly.

\begin{appendix}
\section{\texorpdfstring{Expected total curvature of equilateral closed polygons for
small $n$}{Expected total curvature of equilateral closed polygons for
small $n$}}
\label{appexpectedtotalcurvaturevalues}

In Section~\ref{subsectotalcurvatureofclosedpolygons}, we found
an exact integral formula for the expectation of total curvature for
equilateral $n$-gons following the approach of Grosberg~\cite
{GrosbergExtra}. Grosberg analyzed the asymptotics of this formula for
large numbers of edges,
showing that the expected total curvature approaches the asymptotic
value $n\frac{\pi}{2} + \frac{3\pi}{8}$. We are interested in
evaluating the formula exactly for small $n$ in order to provide a
check on our numerical methods. We used \textit{Mathematica} to evaluate
the formula, obtaining the fractional expressions shown in Table~\ref{tabexpectedtotalcurvature}. Grosberg's asymptotic value is shown in
the rightmost column.

Though for space reasons it had to be truncated in the table, the exact
value for the expected total curvature of equilateral, unconfined
64-gons is

{\fontsize{5.5pt}{7.5pt}\selectfont{\begin{eqnarray*}
&&\!\!\!\!\frac
{4{,}522{,}188{,}530{,}226{,}656{,}504{,}649{,}836{,}292{,}227{,}453{,}294{,}126{,}904{,}427{,}946{,}053{,}625{,}769{,}754{,}177{,}967{,}556{,}412{,}769{,}571{,}113{,}455}
{139{,}655{,}807{,}027{,}685{,}559{,}939{,}231{,}323{,}004{,}419{,}270{,}090{,}691{,}937{,}881{,}733{,}899{,}567{,}960{,}159{,}577{,}537{,}880{,}384{,}373{,}522{,}432}\pi
\\
&&\!\!\!\!\qquad{}+\frac
{288{,}230{,}376{,}151{,}711{,}744}{491{,}901{,}992{,}474{,}628{,}194{,}486{,}464{,}288{,}049{,}342{,}660{,}789{,}103{,}293{,}530{,}486{,}293{,}575{,}717{,}158{,}971{,}541{,}638{,}355{,}891{,}307}.
\end{eqnarray*}}}

\begin{table}
\tabcolsep=0pt
\caption{The expected total curvature of equilateral $n$-gons computed
by evaluating (\protect\ref{eqtotalcurvature}) in \textit{Mathematica} for
$4 \leq n \leq20$ and $n=32, 64$ (the integral becomes singular when
$n=3$, but all triangles have total curvature $2\pi$), together with
Grosberg's asymptotic approximation. We see that for 64-gons we need 5
significant digits to distinguish the exact value from the asymptotic
approximation}
\label{tabexpectedtotalcurvature}
{\fontsize{8pt}{10pt}\selectfont{\begin{tabular*}{\tablewidth}{@{\extracolsep{\fill}}lld{3.5}d{3.5}@{}}
\hline
$\bolds{n}$ & \textbf{Expected total curvature} & \multicolumn
{1}{c}{\textbf{Decimal}} & \multicolumn{1}{c@{}}{\textbf
{Asymptotic}} \\
\hline
\phantom{0}3 & $2\pi$ & 6.28319 & 5.89049 \\[3pt]
\phantom{0}4 & 8 & 8 & 7.46128 \\[3pt]
\phantom{0}5 & $-2 \pi+9 \sqrt{3}$ & 9.30527 & 9.03208 \\[3pt]
\phantom{0}6 & $6 \pi-8$ & 10.8496 & 10.6029 \\[3pt]
\phantom{0}7 & $\frac{316}{33}\pi-\frac{225 }{22} \sqrt{3} $ &
12.369 &
12.1737 \\[4pt]
\phantom{0}8 & $\frac{15}{4} \pi+\frac{32}{15}$ & 13.9143 & 13.7445
\\[4pt]
\phantom{0}9 & $\frac{766}{289} \pi+\frac{11{,}907 }{2890} \sqrt{3}$
& 15.463 &
15.3153 \\[4pt]
10 & $\frac{11}{2} \pi-\frac{64}{245}$ & 17.0175 & 16.8861 \\[4pt]
11 & $\frac{90{,}712}{14{,}219} \pi-\frac{1{,}686{,}177 }{1{,}990{,}660} \sqrt{3}$ &
18.5751 & 18.4569 \\[4pt]
12 & $\frac{331{,}545}{51{,}776} \pi+\frac{512}{28{,}315}$ & 20.1351 &
20.0277 \\[4pt]
13 & $\frac{23{,}336{,}570}{3{,}407{,}523} \pi+\frac{2{,}381{,}643}{22{,}716{,}820} \sqrt
{3}$ & 21.6969 &
21.5984 \\[4pt]
14 & $\frac{877{,}129}{118{,}464} \pi-\frac{1024}{1{,}282{,}743}$ & 23.2601 &
23.1692 \\[4pt]
15 & $\frac{3{,}189{,}814{,}022}{403{,}436{,}289} \pi-\frac
{1{,}786{,}291{,}299}{207{,}097{,}295{,}020} \sqrt{3}$ &
24.8244 & 24.74 \\[4pt]
16 & $\frac{241{,}091{,}487}{28{,}701{,}184} \pi+\frac{4096}{168{,}339{,}171}$ &
26.3896 & 26.3108 \\[4pt]
17 & $\frac{197{,}198{,}281{,}266}{22{,}161{,}558{,}721} \pi+\frac
{44{,}753{,}178{,}051}{88{,}734{,}881{,}118{,}884} \sqrt{3}$
& 27.9554 & 27.8816 \\[4pt]
18 & $\frac{42{,}415{,}625{,}107}{4{,}513{,}689{,}728} \pi-\frac{8192}{15{,}127{,}913{,}229}$ &
29.5219 & 29.4524
\\[4pt]
19 & $\frac{240{,}270{,}145{,}231{,}776
}{24{,}279{,}795{,}663{,}511} \pi-\frac{4{,}277{,}229{,}018{,}201}{194{,}432{,}603{,}673{,}396{,}088} \sqrt
{3}$ & 31.0888 & 31.0232 \\[4pt]
20 & $\frac{111{,}226{,}176{,}353{,}241}{10{,}700{,}200{,}165{,}376} \pi+\frac
{131{,}072}{14{,}288{,}920{,}862{,}931}$ &
32.6561 & 32.594 \\[4pt]
32 & $\frac
{262{,}929{,}167{,}708{,}231{,}675{,}164{,}189{,}486{,}733}{16{,}044{,}875{,}932{,}324{,}628{,}104{,}050{,}900{,}992} \pi
+\frac{134{,}217{,}728}{46{,}358{,}282{,}926{,}117{,}706{,}045{,}930{,}790{,}075}$ & 51.4816 & 51.4436
\\[4pt]
64 & ${\simeq}\frac{4.52218853 \times10^{84}}{1.39655807 \times
10^{83}} \pi+ \frac{2.88230376 \times10^{17}}{4.91901992 \times
10^{80}}$ & 101.7278 & 101.7091\\
\hline
\end{tabular*}}}
\end{table}

\section{\texorpdfstring{Proof of Proposition~\protect\ref{666666666666}}{Proof of Proposition~6}}
\label{sechalfSpaceArms}
In this section, we prove Proposition~\ref{666666666666},
which we restate here.

\begin{proposition}\label{pr34}
The\vspace*{-2pt} polytope
\[
\mathcal{H}_n = \bigl\{\vec{z} \in[-1,1]^n |
z_1 \geq0, z_1 + z_2 \geq 0, \ldots,
z_1 + \cdots+ z_n \geq0, -1 \leq z_i \leq1
\bigr\}
\]
has volume\vspace*{-4pt} $\frac{1}{2^n} {2n\choose n} = \frac{(2n-1)!!}{n!}$.
\end{proposition}

Our proof is a modification of an argument originally suggested on
MathOverflow by Johan W\"astlund~\cite{mo}; Bernardi, Duplantier and
Nadeau~\cite{Bernardi2010ws} seem to have had something similar in\vspace*{-1pt} mind.

\begin{pf*}{Proof of Proposition \ref{pr34}}
Suppose that $s_k(\vec{z}) = z_1 + \cdots+ z_k$ is the $k$th partial
sum of the coordinates of $\vec{z}$, and by convention we set
$s_0(\vec{z}) = 0$. The polytope $\mathcal{H}_n$ can be defined as
the subset of the hypercube where all $s_k(\vec{z}) \geq0$. In the
remainder of the hypercube, the subset of $\vec{z}$ where all the
$s_k(\vec{z})$ are different has full measure: we now partition this
set into a collection of $n$ polytopes $\mathcal{S}_0, \ldots,
\mathcal{S}_{n}$ defined\vspace*{-2pt} by
\[
\mathcal{S}_k := \bigl\{ \vec{z} \in[-1,1]^n -
\mathcal{H}_n|\mbox{the smallest $s_i(\vec{z})$
is $s_k(\vec{z})$}\bigr\}.
\]
We claim that $\operatorname{Vol}\mathcal{S}_k = \operatorname
{Vol}\mathcal{H}_k \cdot\operatorname{Vol}
\mathcal{H}_{n-k}$ for  all $k=1, \ldots, n-1$ and that $\operatorname
{Vol}\mathcal{S}_n =
\operatorname{Vol}\mathcal{H}_n$. Consider the linear\vspace*{-2pt} map
\begin{eqnarray*}
 L_k \dvtx\mathcal{S}_k  &\subset & \mathbb{R}^n
\rightarrow\mathbb {R}^{k} \times\mathbb{R} ^{n-k},
\\[-2pt]
 L_k(z_1,\ldots,z_n) &=& \bigl((-z_k,-z_{k-1},
\ldots, -z_1),(z_{k+1}, \ldots, z_n)\bigr).
\end{eqnarray*}
It is clear that this map preserves unsigned volume. We claim the image
is exactly $\mathcal{H}_k \times\mathcal{H}_{n-k}$. Consider the
partial sums of $(-z_k,\ldots,-z_1)$. The $i$th partial sum is given\vspace*{-2pt} by
\[
s_i(-z_k,\ldots,-z_1) = -z_k -
z_{k-1} - \cdots- z_{k-i+1} = s_{k-i}(z_1,
\ldots,z_n) - s_k(z_1,\ldots,z_n).
\]
The point $(-z_k,\ldots,-z_1)$ is in $\mathcal{H}_k$ $\iff$ this
partial sum is positive for all ${i \in\{1, \ldots, k\}}$. But that
happens exactly when $s_k(\vec{z})$ is negative\footnote{Remember our
convention that $s_0(\vec{z}) = 0$, which is applied when $i=k$.} and
the smallest partial sum among $s_1(\vec{z}), \ldots, s_k(\vec{z})$.
On the other hand, if we consider the partial sums of $(z_{k+1},\ldots
,z_n)$,\vspace*{-3pt} we get
\[
s_i(z_{k+1},\ldots,z_n) = z_{k+1} +
\cdots+ z_{k+i} = s_{k+i}(z_1,
\ldots,z_n) - s_k(z_1,\ldots,z_n).
\]
The point $(z_{k+1}, \ldots, z_n)$ is in $\mathcal{H}_{n-k}$ if and
only if this partial sum is positive for all ${i \in\{1, \ldots, n-k\}
}$. But that happens exactly when $s_k(\vec{z})$ is the smallest
partial sum among $s_k(\vec{z}), \ldots, s_n(\vec{z})$, proving the
claim. When $k=n$, $\mathcal{S}_n$ is just a reversed and negated copy
of $\mathcal{H}_n$ itself.

We now\vspace*{-3pt} have the relation
%
\begin{eqnarray}
\operatorname{Vol}[-1,1]^n &=& 2^n =
\operatorname {Vol}\mathcal{H}_n + \sum\operatorname{Vol}
\mathcal{S}_k
\nonumber
\\[-10pt]
\label{eqHnformula}
\\[-10pt]
\nonumber
&= & 2 \operatorname{Vol}\mathcal{H}_n + \sum
_{k=1}^{n-1} \operatorname{Vol}
\mathcal{H}_k \operatorname{Vol}\mathcal{H}_{n-k}
\end{eqnarray}
and we can prove the formula\vadjust{\goodbreak} by induction on $n$.

When $n=1$, the polytope $\mathcal{H}_1 = [0,1]$ and so the formula
holds. For the inductive step, assume that $\operatorname{Vol}\mathcal
{H}_k = \frac
{1}{2^k} {2k \choose k}$ for all $k < n$. Then solving (\ref{eqHnformula}) for $\operatorname{Vol}\mathcal{H}_n$ yields
%
\begin{equation}
\label{eqinductivestep} \operatorname{Vol}\mathcal{H}_n = 2^{n-1} -
\frac{1}{2^{n+1}}\sum_{k=1}^{n-1} \pmatrix{ 2k \cr k} \pmatrix{2(n-k) \cr n-k}.
\end{equation}
Using the Chu--Vandermonde identity
\[
\sum_{k=0}^n \pmatrix{x \cr k} \pmatrix{y \cr
n-k} = \pmatrix{x+y \cr n}
\]
with $x=y=-\frac{1}{2}$ and recalling that
\[
\pmatrix{-\frac{1}{2} \cr m} = (-1)^m \pmatrix{2m \cr m}
\frac{1}{2^{2m}}  \quad\mbox{and} \quad \pmatrix{-1 \cr p} = (-1)^p
\]
for any positive integers $m$ and $p$, we see that
\[
\sum_{k=1}^{n-1} \pmatrix{2k \cr k} \pmatrix{2(n-k)
\cr n-k} = \sum_{k=0}^n \pmatrix{2k \cr k}
\pmatrix{2(n-k)\cr n-k} -2\pmatrix{2n \cr n} = 2^{2n} - 2\pmatrix{2n \cr n}.
\]
Therefore, equation (\ref{eqinductivestep}) simplifies to
\[
\operatorname{Vol}\mathcal{H}_n = \frac{1}{2^n} \pmatrix{2n \cr n},
\]
as desired.
\end{pf*}
\end{appendix}

\section*{\texorpdfstring{Acknowledgements.}{Acknowledgements}}
We are grateful to many more friends and colleagues for important
discussions related to this project than we can possibly remember to
name here. But to give it our best shot, Michael Usher taught us a
great deal of symplectic geometry, Malcolm Adams introduced us to the
Duistermaat--Heckman theorem, Margaret Symington provided valuable
insight on moment maps, and Alexander Y. Grosberg and Tetsuo Deguchi
have been constant sources of insight and questions on polygon spaces
in statistical physics. Yuanan Diao, Claus Ernst and Uta Ziegler
introduced us to the Rayleigh $\operatorname{sinc}$ integral form for
the p.d.f. of arm
length (and to a great deal more). We were inspired by their insightful
work on confined sampling to look at confinement models above. Ken
Millett and Eric Rawdon have graciously endured our various doubts
about the convergence of the crankshaft and fold algorithms for many
years, and were the source of many pivotal conversations. Chris Soteros
provided much appreciated expert guidance on Markov chain sampling.
Jorge Calvo, Kate Hake and Teresita Ramirez-Rosas read the draft
extremely carefully and made some helpful corrections. And we are
especially indebted to Alessia Mandini, Chris Manon, Angela Gibney and
Danny Krashen for explaining to us some of the elements of the
algebraic geometry of polygon spaces.

We are also deeply appreciative of the efforts of the editor, associate
editor and referees, who made excellent suggestions for improving this paper.

We were supported by the Georgia Topology Conference Grant
DMS-11-05699, which helped us organize a conference on polygon spaces
in the summer of 2013.
We are grateful to the Issac Newton Institute
for the Mathematical Sciences, Cambridge, for support and hospitality
during the program ``Topological Dynamics in the Physical and
Biological Sciences'' in Fall 2012, when much of this work was completed.






\printaddresses
\end{document}